\documentclass[pdflatex]{sn-jnl}

\makeatother

\usepackage{tikz-cd}

\usepackage{graphicx}
\usepackage[numbers]{natbib}
\usepackage{graphicx}%
\usepackage{multirow}%
\usepackage{amsmath,amssymb,amsfonts}%
\usepackage{amsthm}%
\usepackage{mathrsfs}%
\usepackage{rotating}%
\usepackage{appendix}%
\usepackage{textcomp}%
\usepackage{manyfoot}%
\usepackage{booktabs}%
\usepackage{etoolbox}
\usepackage{algorithm}%
\usepackage{xcolor}%
\usepackage{algorithmicx}%
\usepackage{algpseudocode}%
\usepackage{listings}%
\usepackage{url}

\newcommand{\bx}{\boldsymbol{x}}
\newcommand{\by}{\boldsymbol{y}}
\newcommand{\bX}{\boldsymbol{X}}

\newtheorem{lemma}{Lemma}
\newtheorem{example}{Example}
\newtheorem{theorem}{Theorem}
\newtheorem{proposition}[theorem]{Proposition}%

\begin{document}
\title[Article title]{Fisher-Rao distance between truncated distributions and robustness analysis in uncertainty quantification}

\author*[1,2]{\fnm{Baalu Belay} \sur{Ketema}}\email{baalu-belay.ketema@edf.fr}
\author[2,3,4]{\fnm{Nicolas} \sur{Bousquet}} 
\author[2]{\fnm{Francesco} \sur{Costantino}}
\author[2]{\fnm{Fabrice} \sur{Gamboa}}
\author[1,2,4]{\fnm{Bertrand} \sur{Iooss}}
\author[1]{\fnm{Roman} \sur{Sueur}}

\affil[1]{\orgname{EDF R\&D}, \orgaddress{\street{6 quai Watier}, \postcode{78401}, \city{Chatou}, \country{France}}}

\affil[2]{\orgdiv{IMT}, \orgname{Universit\'e Paul Sabatier}, \orgaddress{\street{118 Route de Narbonne}, \city{Toulouse}, \postcode{31062}, \country{France}}}

\affil[3]{\orgdiv{LPSM}, \orgname{Sorbonne Universit\'e}, \orgaddress{\street{4 place Jussieu}, \city{Paris}, \postcode{75005}, \country{France}}}

\affil[4]{\orgdiv{SINCLAIR AI Laboratory, EDF R\&D}, \orgname{Sorbonne Universit\'e}, \orgaddress{\city{Palaiseau}, \postcode{91120}, \country{France}}}

\abstract{Input variables in numerical models are often subject to several levels of uncertainty, usually modeled by probability distributions. In the context of uncertainty quantification applied to these models, studying the robustness of output quantities with respect to the input distributions requires: (a) defining variational classes for these distributions; (b) calculating boundary values for the output quantities of interest with respect to these variational classes. The latter should be defined in a way that is consistent with the information structure defined by the “baseline” choice of input distributions. Considering parametric families, the variational classes are defined using the geodesic distance in their Riemannian manifold, a generic approach to such problems. Theoretical results and application tools are provided to justify and facilitate the implementation of such robustness studies, in concrete situations where these distributions are truncated -- a setting frequently encountered in applications. The feasibility of our approach is illustrated in a simplified industrial case study.}


\maketitle

\section{Introduction}

In applied mathematics, the importance of uncertainty quantification (UQ) stems from the considerable development of engineering and decision support methods. Generally, these methods are based on numerical models or models learned from data (not to mention hybrid ones) and provide predictive diagnoses. Robustness analyses (RA) form an important part of such studies \cite{idrissi2022quantileconstrained,DaVeiga2021}.
Indeed, the goal of RA is to quantify the impact of the modeling assumptions in the model input variables.
This step is crucial in many applications, for example when a certification from authorities is required \cite{largau24}.

Very generically, RA  can be  formalized as follows. Denoting $G(\textbf{X})=Y$ a deterministic model of interest with inputs $\boldsymbol{X}:=(X_1,...,X_d)$ and output $Y$, $\boldsymbol{X}$ is modeled as a random variable with distribution $P_{\textbf{X}}$ on $\mathbb R^d$ (either theoretical or empirical).
Let us define  some quantity of interest (QoI) on $Y$. For instance, many authors in structural reliability consider a unidimensional output $Y$ and define $\text{QoI}(Y)$ as \textcolor{black}{a function of a probability}, quantile or superquantile of $Y$ induced by $P_\textbf{X}$ given $G$ \cite{AJENJO2022102196,Lemaitre2015,MAUMEDESCHAMPS2018122,Sueur2017,iooss2022bepu}.  Often $P_\textbf{X}$ suffers from  \emph{misspecification}.  
Indeed, the assessment of a statistical distribution for $\boldsymbol{X}$ from a mixture of observations and external knowledge, a current approach in many engineering fields \cite{dienstfrey2012uncertainty,Hanea2022}, is subject to 
modeling errors. 
The stacking of these errors is  generally interpreted as  so-called \emph{epistemic uncertainties} \cite{Helton1996,Hullermeier2021}. 
An empirical vision of this problem, in a machine learning framework, is known as the domain shift problem \cite{Turhan2012}. 
The aim of RA  studies is therefore to check, through some \emph{robustness indices}, whether $\text{QoI}(Y)$ remains weakly dependent on this misspecification of $P_\textbf{X}$, by making $P_\textbf{X}$ vary in a given \textcolor{black}{\emph{variational class} ${\cal{C}}$}, a subspace of the space ${\cal{P}}$ of feasible distributions on $\boldsymbol{X}$.
Therefore, following the semantics proposed by \cite{Lemaitre2015}, RA studies are based on  \emph{distributional perturbation} methods.

Defining ${\cal{C}}$ was the subject of many works in recent years \cite{AJENJO2022102196}.
A major contribution to the field was made by \cite{Gauchy2022}, providing considerable motivation for our work. In this paper, the inputs $X_1,...,X_d$ are assumed \textcolor{black}{independent} and $P_\textbf{X}$ varies only through its $d$ marginals $P_{X_i}$ in a certain family $\mathcal P_i$ of distributions on $\mathbb R$. The authors then argue that a generic RA with respect to (w.r.t.) input domain misspecification requires to handle the information geometry of each $\mathcal P_i$ 
through its Fisher-Rao (FR) distance $d_i$ \cite{rao1945information}. This requires that each $\mathcal P_i$ is 
a parametric family $\mathcal P_i=\{P_{X_i,\theta}\}_{\theta\in \Theta_i}$ of distributions on $\mathbb R$ containing $P_{X_i}$ and with well-defined Fisher information. The authors also explain that a good candidate for the variational class $\mathcal C$ is given by a collection of concentric FR spheres $\Lambda_{i\delta}$ with radii $\delta\in(0,\delta_{\max}]$ centered at the marginal distribution $P_{X_i}$, for each $i$. Then, they propose to compute robustness indices based on the maximum and minimum values of a QoI over these FR spheres $\Lambda_{i\delta}$.  

This approach promotes a more objective RA framework than the perturbations of probability distributions based on entropy criteria,  that involve inherently subjective form constraints, as generalized moment constraints proposed in \cite{Bachoc2020,Lemaitre2015}. Besides, perturbations based on the intrinsic geometry of distributions are by essence multivariate, which allows to go beyond RA based on optimal transport techniques modifying marginal percentiles  \cite{idrissi2022quantileconstrained}. 

The proposal made by \cite{Gauchy2022} is part of a slow but progressive trend in UQ to manipulate Riemannian and information geometry.
Information theory was used in \cite{ludtke2008information} as early as 2008 
to define sensitivity indices based on input-output mutual information. One of the first use of information geometry in UQ may be found in \cite{sternfels2013geometric} for building reduced order models from snapshots. Other works in this field followed since, for instance in Bayesian UQ \cite{house2016bayesian}. Very recently, a symplectic decomposition of the Fisher information matrix is used for parameter sensitivity analysis in \cite{yang2024decision} while information geometric techniques are put in action for studying sloppy models in \cite{kurniawan2022bayesian}. But up to our knowledge, the only work proposing an information geometry-based method for RA applied to numerical models, handling two-parameter distributions (but suffering from poor accuracy performance linked to Monte Carlo approximations of Fisher information) remains \cite{Gauchy2022}. 

This relatively slow adoption by the UQ community is probably due to: (a) a lack of knowledge about the required geometrical tools; and (b) a lack of known results (except those produced in \cite{Gauchy2022}) concerning the theoretical features and practical calculations, for concrete problems, of parametric distributions defining ${\cal{C}}$.

This twofold need is the main motivation of our work. Let us discuss the results below. First, we provide an illustration of the  method and highlight the explicit correspondence between the FR spheres $\Lambda_{i\delta}$  and distributional perturbations on selected families of 
probability distributions (normal, log-normal, Gumbel, Gamma, triangular, exponential, etc.).  Second, we obtain new explicit computations on Fisher information and Christoffel symbols for a truncated version of the normal distributions. This is of particular interest as they are often used in applications \cite{nechval2016novel}. This allows us to derive the corresponding FR distances. Lastly, we  present more general results for handling two common classes of families of distributions, namely location-scale and pushforward families.

More precisely, we obtain these results by taking the following path. We first propose in Section \ref{sec:information.geometry} an initiation or reminder on information geometry, in a classical theoretical setting. 
This description is adapted to the special case of truncated distributions in Section \ref{sec:RA framework and truncated distributions}, where we also detail the RA framework, explain the distributional perturbation method and illustrate the latter on the normal family in usual and truncated cases. Section \ref{sec: Fisher-Rao spheres illustration} gives the main algorithm for computing the FR spheres $\Lambda_\delta$ and illustrates the perturbation method on the families mentioned above. A numerical experiment is proposed in Section \ref{sec:numerical experiment}, describing the 
optimization procedure and illustrating the feasibility of robustness analyses, through the study of a simplified flood model. Besides, a discussion section closes our paper, opening towards several research perspectives aimed at extending the use of this type of robustness analysis in UQ problems.

\textcolor{black}{We postpone to Appendix \ref{appendix: statistics} some technical details on the statistical procedure for estimating (with bootstrap-based confidence intervals) the robustness indices proposed in \cite{Gauchy2022}. Proofs and additional geometric explanations are given in Appendix \ref{appendix: geometric results}. Some tedious technical computations and explanations on the Fisher information and Christoffel symbols are provided in a Supplementary Material. 
Finally, Table \ref{tab:acronyms} provides the main acronyms used in this paper.}

\begin{table}[h!]
    \centering
    \begin{tabular}{cc}
        \hline
        UQ & Uncertainty quantification\\
        \hline
        RA & Robustness analysis\\
        \hline
        QoI & Quantity of interest \\
        \hline
        FR & Fisher-Rao\\
        \hline
        FIM & Fisher information matrix \\
        \hline
        pdf & probability density function\\
        \hline
    \end{tabular}
    \caption{Main acronyms}
    \label{tab:acronyms}
\end{table}

\section{Riemannian geometry of probability distributions}\label{sec:information.geometry}

Since the pioneering works by Hotelling \cite{hotelling1930spaces}, Rao \cite{rao1945information} and Jeffreys \cite{jeffreys1946invariant}, exploring the geometry of \textcolor{black}{parametric families of distributions} has become a very active reseach area, 
which has notably produced new interpretations of the mechanisms of statistical inference \cite{efron1975defining,kass1989geometry}. However, this framework often remains unfamiliar to UQ practitioners. 
Therefore basic concepts and tools of Riemannian and information geometry -- which can be skipped on first reading -- are briefly presented in the following paragraphs. Readers interested in more details are referred to  \cite{amari2000methods}, \cite{calin2014geometric}, \cite{do1992riemannian} and \cite{nielsen2020elementary}.

\subsection{Basic notions}\label{sec: basic notions in riem geom}

Riemannian manifolds are smooth manifolds $M$ equipped with a metric tensor usually denoted $g$ (page 38 in \cite{do1992riemannian}, Definition 2.1). \textcolor{black}{At a given base point $p\in M$, this metric tensor denoted $g_p$ is a scalar product on each tangent space $T_p M$ of $M$ and varies smoothly w.r.t. the base point $p$}. This metric allows to generalize on manifolds many  geometric notions like angles, norms and distances already known on the 2D plane or 3D space. For instance, the distance between two points $\bx$ and $\by$ on $M$ is defined as the length of the shortest paths connecting $\bx$ and $\by$\begin{align} \label{eq: riem dist}
d(\bx,\by):= \inf_{\alpha(0) = \bx, \  \alpha(1) = \by}   \ell(\alpha)
\end{align}
where $\ell(\alpha)$ is the length of the curve $\alpha:[0,1] \to M$ defined as
$$\ell(\alpha) = \int_0^1 |\dot \alpha(\tau)|_{\alpha(\tau)}  d\tau$$
and $|v|_{\bx} $ is the norm of $v\in T_{\bx} M$ defined as $|v|_{\bx}= \sqrt{g_{\bx}(v,v)}$. This distance is called the Riemannian geodesic distance and it is intrinsic to the metric $g$. 

Affine connections, an important class of objects in geometry, are differential operators that allow to compute directional derivatives of vector fields. On Riemannian manifolds there exists a unique one, called the Levi-Civita connection and denoted $\bar \nabla$, that is compatible to the metric (and torsion-free). \textcolor{black}{Roughly speaking, the connection $\bar \nabla$ provides a natural extension of the classical formula for differentiating a scalar product (see Definition 3.1 page 53 in \cite{do1992riemannian}).} Denoting $m$ the dimension of $M$, at each point $p\in M$ this connection is encoded by $\frac{m^2(m+1)}{2}$ coefficients $\bar \Gamma_{ij}^k(p)$ for $1\leq i,j,k\leq m$ called \emph{Christoffel symbols}, which depend on a local chart around $p$, say $x^1,...,x^m$. In this chart, these symbols are given by
\begin{align} \label{Christoffel symbols}
    \bar \Gamma_{ij}^k  = \sum_{l=1}^m\frac{g^{kl}}{2} \left( \frac{\partial g_{il}}{\partial x^j} + \frac{\partial g_{jl}}{\partial x^i} - \frac{\partial g_{ij}}{\partial x^l}  \right),
\end{align}  
where $g_{ij}$ and $g^{ij}$ correspond resp. to the $ij$ coefficient of the matrix of $g$ and its inverse and $\frac{\partial g_{ij}}{\partial x^l}$ are the partial derivatives of $g_{ij}$  in the chart $x^1,...,x^m$. All these quantities are evaluated at point $p\in M$.

Affine connections allow us to define the second derivative, or ``acceleration", of a curve $\gamma: [0,1] \to M$ 
by seeing $t \mapsto\dot \gamma_t$ as a vector field along $\gamma$ and computing its directional derivative at point $\gamma_t$ and in the direction $\dot \gamma_t$ as
$$\bar \nabla_{\dot \gamma_t} \dot \gamma_t \in T_{\gamma_t} M,$$
which can be thought as the derivative of $\dot \gamma$ i.e. the second derivative of $\gamma$. Curves that have no ``acceleration" are called \emph{geodesics} and they solve a second order differential equation given by $\bar \nabla_{\dot \gamma} \dot \gamma = 0$, or in coordinates (with Einstein summation convention)
\begin{align}\label{geodesic equation}
    \ddot \gamma^k + \dot \gamma^i \bar\Gamma_{ij}^k \dot \gamma^j = 0, \ \ \text{for all } k=1,...,m.
\end{align}  

Geodesics $\gamma$ locally minimize the Riemannian distance (\ref{eq: riem dist}) since for small $t>0$ 
\begin{align*} 
  d(\gamma_0,\gamma_t) = \int_0^t |\dot \gamma_\tau|_{\gamma_\tau}  d\tau = t|\dot \gamma_0|_{\gamma_0}.  
\end{align*}
This shows that  
geodesics $\gamma$ are constant speed curves,  
since the distance between $\gamma_0$ and $\gamma_t$ is given by the initial speed $|\dot \gamma_0|_{\gamma_0}$ multiplied by the travel time $t$.

\subsection{Geometry on a family of parametric probability distributions}
Given a parametric family $\mathcal P = \{P_\theta\}_{\theta\in\Theta}$ 
of probability measures, with $\Theta$ an open set in $\mathbb R^d$, the seminal idea proposed by Hotelling and Rao \cite{rao1945information,nielsen2020elementary} is to describe $\mathcal P$ as a Riemannian manifold endowed with the Riemannian metric given by the Fisher Information matrix (FIM) $I_\theta$, whose components are 
\begin{align}\label{eq:coord Fisher}
(I_\theta)_{ij} = \mathbb E_{X\sim P_\theta} \Big[\partial_i \log p_\theta(X)\partial_j \log p_\theta(X)\Big] 
\end{align}
where $p_\theta = {dP_\theta}/{d\mu}$ is a probability density function (pdf) of $P_\theta$ w.r.t. a reference measure $\nu$ and $\partial_i \log p_\theta(x) = \frac{\partial}{\partial \theta^i} \log p_\theta(x)$.

A few regularity conditions are usually assumed to ensure the FIM exists and is positive semi-definite. For instance $\mbox{Supp}(P_\theta)$ should be independent from $\theta\in\Theta$;  
the map $\theta \mapsto p_\theta(x)$ should be smooth;  and 
the integration $\mathbb E_{X\sim P_\theta}(\cdot)$ and differentiation $\frac{\partial}{\partial \theta^i}$ operations should be interchangeable. 
In addition, linear independence of the functions $\{x \mapsto \partial_{i} p_\theta(x)\}_{1\leq i\leq d}$ for all $\theta$ guarantees the positive-definiteness of $I_\theta$. 
More details on the required regularity assumptions are provided in \cite{lehmann2006theory,calin2014geometric,amari2000methods}. 

Under these assumptions, the FIM $I_\theta$ is the matrix representation of a scalar product $g_\theta$ on the tangent space of $\mathcal P$ at $P_\theta$, usually called Fisher metric. 
Assuming the smoothness of (\ref{eq:coord Fisher}) w.r.t. $\theta$, the couple $(\mathcal P,g)$ becomes a Riemannian manifold. This allows to conduct geometrical analyses on $\mathcal P$ using the intrinsic Riemannian distance (\ref{eq: riem dist}), geodesic curves (\ref{geodesic equation}) and many more tools.
With our previous notations, the manifold is $M:=\mathcal P$ and the metric tensor is $g_\theta$.

For a parametric family of distributions $\mathcal P$, the Riemannian geodesic distance $d$ given by (\ref{eq: riem dist}) is called the FR distance. This distance is intrinsic to the family $\mathcal P$ and is invariant under reparameterization 
\cite{calin2014geometric}. In addition, for $P_\theta$ in the vicinity of $P_{\theta_0}$, the FR distance behaves like the Kullback-Leibler divergence $\mathcal D$: 
$$\mathcal D(P_\theta | P_{\theta_0}) = \frac{1}{2} d(P_\theta,P_{\theta_0})^2 + o\big(d(P_\theta,P_{\theta_0})^2\big),$$
where $\mathcal D(P_\theta| P_{\theta_0}):= \mathbb E_{X\sim P_\theta} \left[\log\left({P_\theta}/{P_{\theta_0}}(X)\right)\right]$. However, a global statistical interpretation for $d$ is less obvious. 
Computing an approximation of $d$ is still possible if the Christoffel symbols (\ref{Christoffel symbols}) are available. This will be further discussed in Section \ref{sec: Fisher-Rao spheres illustration}.

\section{Robustness analysis and truncated probability distributions}\label{sec:RA framework and truncated distributions}

\subsection{Robustness analysis framework }\label{sec:RA.framework}
Echoing the concepts and notations introduced in the first section, let us consider the RA framework \cite{Gauchy2022} defined by:
\begin{itemize}
    \item A QoI on the stochastic variable $Y:=G(\bX)$ where $G$ is a deterministic forward model and
    $\bX = (X_1,\ldots,X_d)$ is a random vector with distribution $P_{\bX}$;
    \item A family $\mathcal P_i$, for each input $X_i$, of probability distributions each containing the \emph{baseline} marginal $P_{X_i}$ of $P_{\bX}$, with pdf $f_i$;
    \item A distance $d_i$ on each family $\mathcal P_i$, that will be precised in the sequel. This allows to define the $\delta$-\textit{perturbations} of $f_i$ as any other element  $f_{i\delta}\in\mathcal P_i$ at $d_i$-distance $\delta$ from $f_i$. The collection of such elements is denoted by $\Lambda_{i\delta}$ and forms a metric sphere in $\mathcal P_i$; 
    
    \item A set $\{S_i\}_{1\leq i \leq d}$  of robustness indices, as for example the perturbed law-based indices (PLI ; \cite{Lemaitre2015}) defined for nonzero QoI:
    \begin{align}
        S_i(f_{i\delta})  =  \frac{\text{QoI}(Y^{i\delta})}{\text{QoI}(Y)} - 1, \label{eq:PLI}
    \end{align}
    where $Y^{i\delta}:=G(\textbf{X}^{i\delta})$ for $\textbf{X}^{i\delta}\sim P_{\boldsymbol{X}^{i\delta}}$ and $P_{\boldsymbol{X}^{i\delta}}$ is the joint distribution of $\bX$ for which only the $i$-th marginal pdf $f_i$ has been replaced by some $f_{i\delta}\in \Lambda_{i\delta}$. \textcolor{black}{In other words, $P_{\boldsymbol{X}^{i\delta}}$ has density $f_1\otimes \ldots \otimes f_{i\delta} \otimes \ldots \otimes f_d$.}
\end{itemize}

Having set the RA framework, we now explain how the RA study of $G$ is converted into a mathematical problem. We then precise the distance $d_i$ used on $\mathcal P_i$ before defining truncated distributions and explaining our interest in such families for RA. The $\delta$-perturbations for distance $d_i$ are illustrated in the normal family for both the truncated and non-truncated case. 

The indices $S_i$ can be seen as functions $f_{i\delta} \in \mathcal P_i \mapsto S_i(f_{i\delta}) \in \mathbb R$ quantifying the robustness of $\text{QoI}(Y)$ w.r.t. perturbations on $f_i$. A value close to zero for $S_i(f_{i\delta})$ means that $\text{QoI}(Y)$ is not significantly affected when perturbing $f_i$ to $f_{i\delta}$, reflecting its robustness w.r.t. the distribution of $X_i$. The following bounds or \textcolor{black}{\emph{synthetic indices}}, introduced in \cite{Gauchy2022}, quantify the worst impact of such perturbations given a $\delta$ ranging in $(0,\delta_{\max}]$:
\begin{align*}
    S_{i\delta}^-:=\min_{f_{i\delta} \in \Lambda_{i\delta}} S_i(f_{i\delta}) \ \ \ \ \text{and} \ \ \ S_{i\delta}^+:=\max_{f_{i\delta} \in \Lambda_{i\delta}} S_i(f_{i\delta}) \tag{$\star$}.
\end{align*}
Ultimately, a maximum level $\delta_{\max}$ must be elicited to get a practically usable methodology, which remains an open topic (see Section \ref{sec: discussion}).

In practice, the robustness indices $\{S_i(f_{i\delta})\}_i$ cannot be computed explicitly
and must be estimated using evaluations from the model $G$. Details on the estimation procedure used in \cite{Lemaitre2015, Sueur2017, iooss2022bepu, Gauchy2022} are provided in Appendix \ref{appendix: estimation procedure}. Similarly, the resolution of $(\star)$ cannot
be achieved through usual gradient-based optimization techniques and thus requires a dedicated
algorithm.

\subsection{Defining the perturbation sets 
} \label{subsec: def perturbation}

Defining how the baseline pdf (or equivalently distribution) $f_i$ should be  perturbed has been a research subject considered by numerous authors \cite{AJENJO2022102196,defaux2018efficient,Lemaitre2015}. 
In this article, for the sake of \textcolor{black}{genericity} and following the recommendations given in \cite{Gauchy2022}, we assume $\mathcal P_i$ is a parametric family $\mathcal P_{i}=\{f_{i,\theta}\}_{\theta \in \Theta_i}$ containing $f_i$.
Each $\mathcal P_{i}$ is endowed with the Fisher metric as explained in Section \ref{sec:information.geometry} and $d_i$ is taken as the corresponding FR distance. Accordingly, the subsets $\{\Lambda_{i\delta}\}_{\delta \leq \delta_{\max}}$ are defined as concentric FR spheres centered at $f_i$ with radii $\delta\in (0,\delta_{\max}]$:
$$\Lambda_{i\delta}:= \{f_{i,\theta} \in \mathcal P_{i} \ | \ d_i(f_i,f_{i,\theta})=\delta\}.$$

Let us illustrate this distributional perturbation method on $\mathcal P:=\{\mathcal N(\mu,\sigma^2)\}_{(\mu,\sigma)\in\Theta}$ the normal family by computing the spheres $\Lambda_{\delta}$, the exact method for such computations is detailed in Section \ref{sec: Fisher-Rao spheres illustration}. Let us first point out that the FIM in the $(\mu,\sigma)$-parameterization is easily obtained (see \cite{nielsen2020elementary}). Indeed, we have
$$ I_\theta = \frac{1}{\sigma^2}\displaystyle \begin{bmatrix}
    1 & 0 \\ 0 & 2
\end{bmatrix}, \ \ \ \theta=(\mu,\sigma) \in \Theta,$$
which defines the same geometry as the so-called Poincaré metric on the Poincaré half-plane $\mathbb H:=\{x+iy \ | \ y>0\}$ (see \cite{costa2015fisher}, as well as Section \ref{sec: loc-scale families} below on location-scale families). Most geometric quantities can thus be computed explicitly. This is not the case in general. For instance the Christoffel symbols $\bar \Gamma_{ij}^k$ (\ref{Christoffel symbols}) are given by
$$ \bar \Gamma_{12}^1=\bar \Gamma_{21}^1 =\bar \Gamma_{22}^2= \frac{-1}{\sigma}, \ \ \ \bar \Gamma_{11}^2=\frac{1}{2\sigma}, \ \ \ \bar \Gamma_{11}^1 = \bar \Gamma_{22}^1 = \bar \Gamma_{21}^2 = \bar \Gamma_{12}^2 = 0.$$

Figure \ref{fig: non truncated normal family} gives an illustration of the concentric FR spheres and $\delta$~-~perturbations of a normal density function.

\begin{figure}[!ht]
     \centering
     \begin{tabular}{ll}
     \begin{minipage}{0.5\textwidth}
         \includegraphics[scale=0.22]{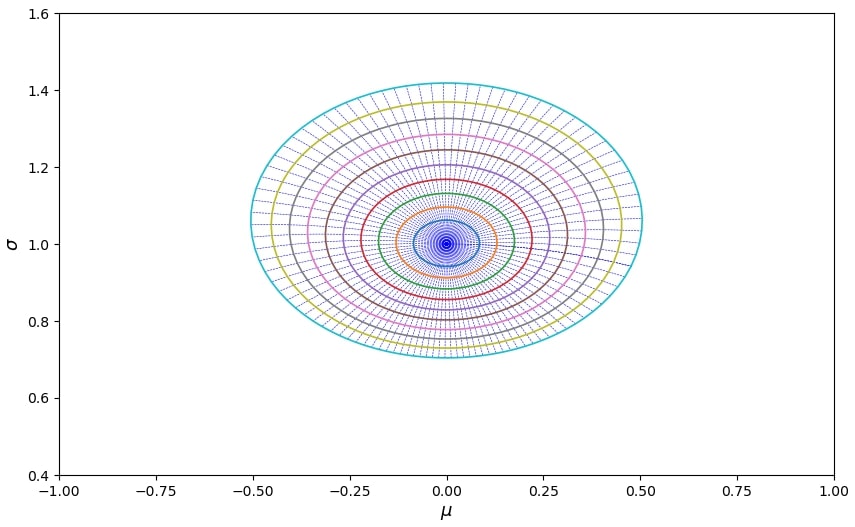}
    
    \footnotesize Concentric FR spheres $\Lambda_{\delta}$ centered at $\mathcal N(0,1)$ for radii $\delta \leq 1$ in the $(\mu,\sigma)$ parameter space.
    \end{minipage}
     & 
     \begin{minipage}{0.5\textwidth}
         \includegraphics[scale=0.15]{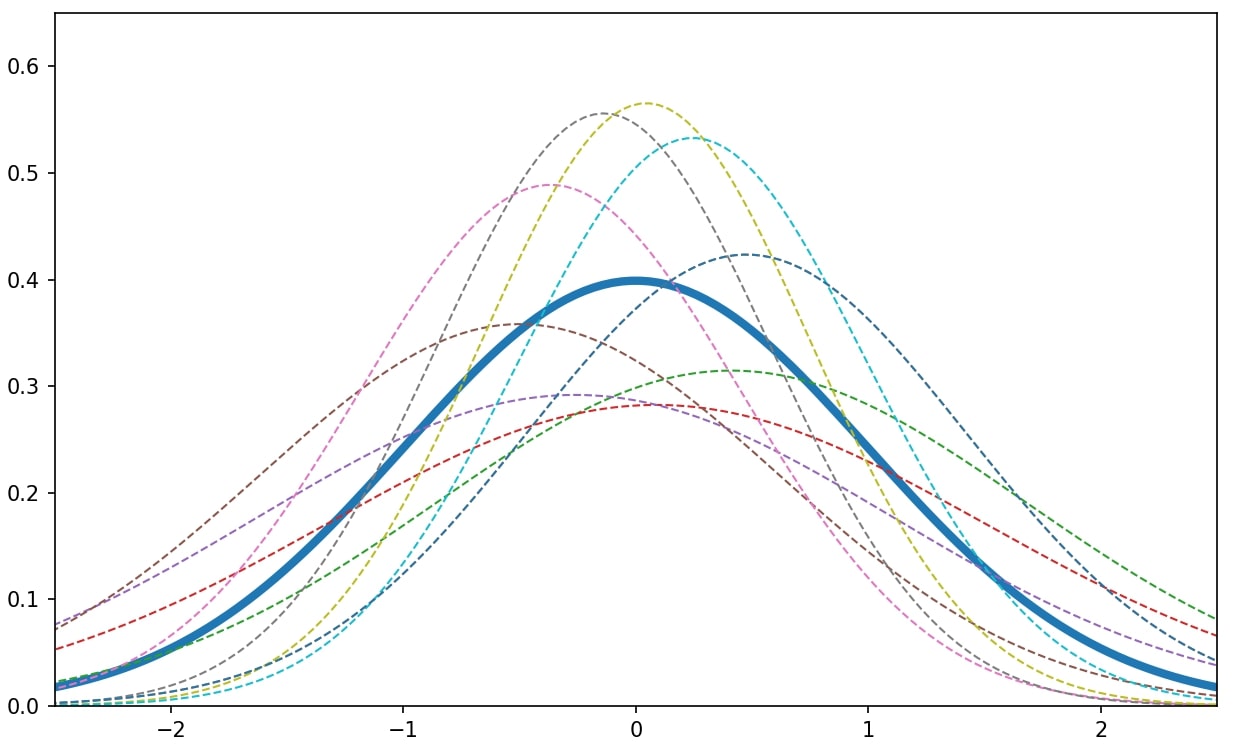}
         
         \footnotesize Some $\delta$-perturbations (in dashed) of the $\mathcal N(0,1)$ density (in solid) i.e. elements of $\Lambda_{\delta}$ for $\delta=1$.
    
    \end{minipage} \\
    \end{tabular}
    \caption{$\delta$~-~perturbations of a normal density function.}
    \label{fig: non truncated normal family}
\end{figure}

\subsection{Robustness analysis for truncated distributions}

Specifying the RA framework for truncated marginals of $P_{\boldsymbol{X}}$ appears particularly 
relevant in many engineering problems (e.g. in risk, reliability and safety issues \cite{Du2012}). 

In UQ problems, uncertain (stochastic) inputs of computer models $G$ used for such analyses are often considered on a restriction of their support, focusing on the most critical situations (i.e., leading to critical output values $Y$). For instance, in \cite{Gauchy2022} the authors consider two case studies in flood and nuclear risks, each involving computer models where the input distributions (e.g., Gumbel, normal, log-normal) are truncated on sub-domains that are most likely to lead to undesired output situations. Such studies, that typically aim to test designs, are obviously concerned with robustness analyses. Other truncated distributions can originate from the physical limitations of sensors \cite{Ding2016} or to respect the physical meaning of input variables (e.g., in hydrology \cite{Shao2023}, wind energy production \cite{PAPI2021701} or wind lidar calibration \cite{zhang2022high}).    

Therefore, we focus on a 1D marginal component $X$ of $\bX$, and consider interval truncation. 
Assuming that the possible distributions of $X$ are modeled by a parametric family $\mathcal P:=\{
P_\theta\}_{\theta \in \Theta}$ with densities $f_\theta$, the \textit{truncated distributions} of $X\sim P_\theta$ on an interval $[a,b]$ are defined as the conditional probability distributions $Q_\theta$ of $X\sim P_\theta$ over $[a,b]$ with corresponding pdf $q_\theta$ resp. given by
$$Q_\theta(A):=\frac{P_\theta(A\cap[a,b])}{P_\theta([a,b])} \ \ \ \text{ and } \ \ \ q_\theta(x) := \frac{f_\theta(x) }{P_\theta([a,b])}\textbf{1}_{x\in [a,b]}.$$
We can therefore define the FIM, the FR distance and other geometric quantities on the truncated family $\mathcal Q:=\{Q_\theta\}_{\theta\in\Theta}$.

Note that beyond the interval truncation considered here, truncated marginal distributions can be defined on any measurable subset $B\subset \mathbb R$ as well as for densities on $\mathbb R^d$.

Before moving on to the next section, let us illustrate this by considering a truncated version of the normal family on $[a,b]$, denoted $\mathcal N_{[a,b]}$, whose pdfs are given by
$$q_\theta(x) = \frac{1}{N_\theta} \frac{1}{\sqrt{2\pi}\sigma} \exp\left(-\frac{(x-\mu)^2}{2\sigma^2}\right)\textbf{1}_{x\in[a,b]}, \ \ \ \text{for } \theta = (\mu,\sigma),$$
where $N_{\theta}=(\sqrt{2\pi}\sigma)^{-1}\int_a^b  \exp\big(-(x-\mu)^2/(2\sigma^2)\big)\ dx $ is the normalizing constant. We proved  the following result. 

\begin{proposition} \label{prop: tr normal FIM}
    The FIM for the $\mathcal N_{[a,b]}$ family is given by
\begin{align*}
I_\theta^{[a,b]} = \frac{1}{\sigma^2}\begin{bmatrix}
     \displaystyle \partial_\mu \mu_{[a,b]}  & \displaystyle \partial_\sigma \mu_{[a,b]}\\
     \displaystyle \partial_\sigma \mu_{[a,b]} & \displaystyle \sigma^{-1} \big( \partial_\sigma \sigma_{[a,b]}^2 + 2\big(\mu_{[a,b]} - \mu\big) \partial_\sigma \mu_{[a,b]} \big)
\end{bmatrix},
\end{align*}
where $\mu_{[a,b]}$ and $\sigma_{[a,b]}^2$ are functions of $\theta=(\mu,\sigma)$ given resp. by the expectation and variance of $q_\theta$.
\end{proposition}
The exact computation of the FIM coefficients and Christoffel symbols 
is detailed in the online supplementary materials. 
FR Spheres $\Lambda_\delta$ and $\delta$-perturbations in the truncated family $\mathcal N_{[-2,2]}$ are illustrated in Figure \ref{fig: truncated normal family}.

\begin{figure}[!ht]
     \centering
     \begin{tabular}{ll}
     \begin{minipage}{0.5\textwidth}
         \hspace{-0.5cm}\includegraphics[scale=0.22]{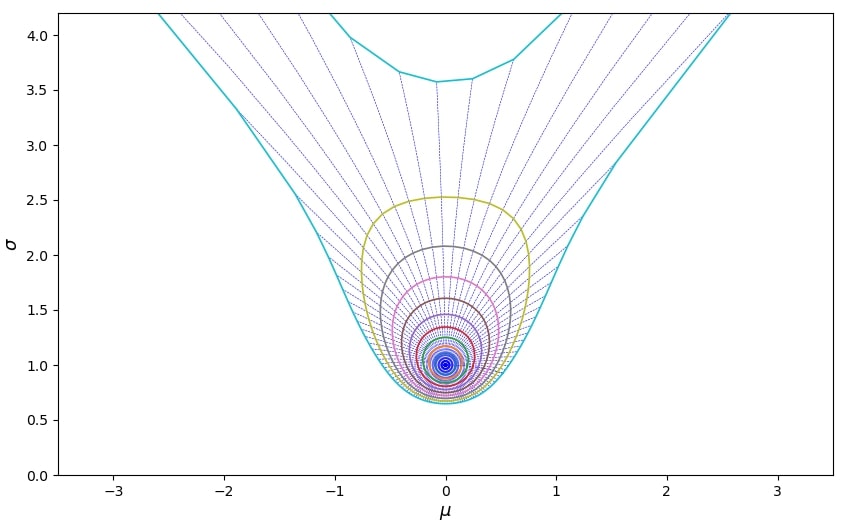}

    \footnotesize Concentric FR spheres $\Lambda_\delta$ centered at $\mathcal N_{[-2,2]}(0,1)$ with radii $\delta \leq 0.5$ in the $(\mu,\sigma)$ parameter space. 
    \end{minipage}
     & 
     \begin{minipage}{0.5\textwidth}
         \hspace{-0.5cm}\includegraphics[scale=0.15]{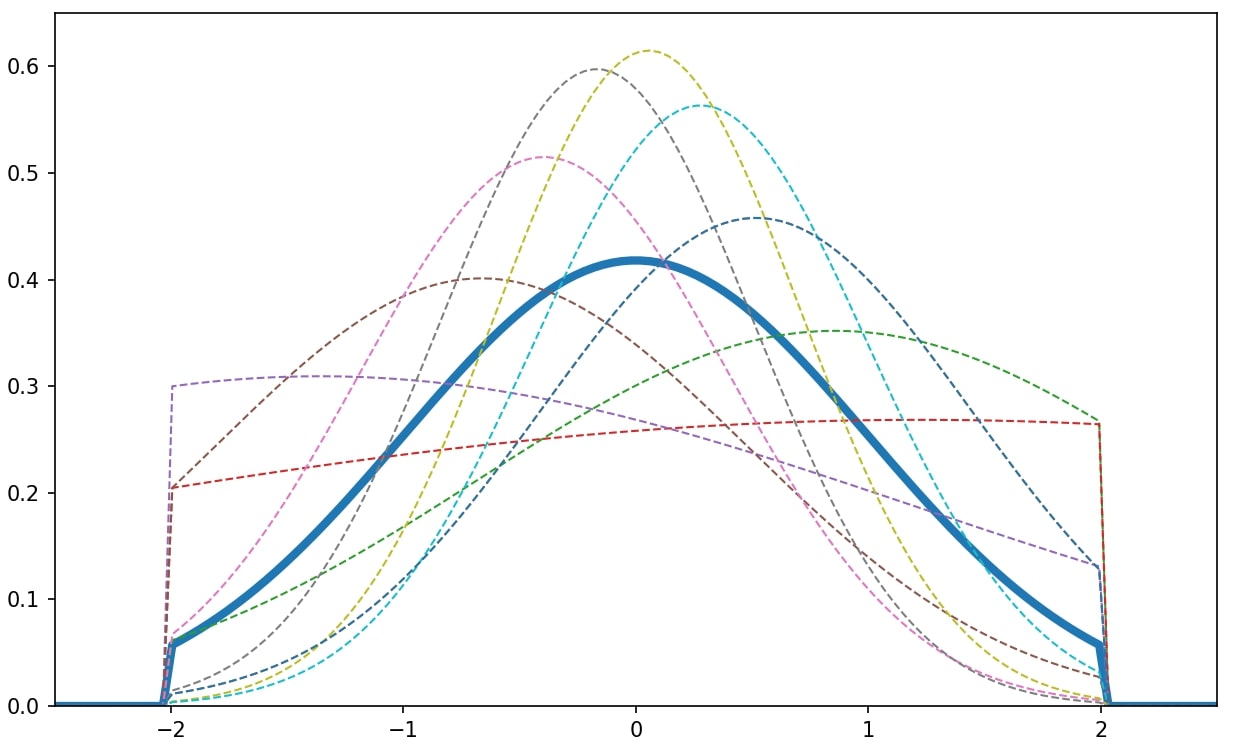}

         \footnotesize Some $\delta$-perturbations (in dashed) of the $\mathcal N_{[-2,2]}(0,1)$  density (in solid) i.e. elements of $\Lambda_\delta$ for $\delta=0.5$
    \end{minipage} \\
    \end{tabular}
    \caption{Illustration of the $\delta$-perturbation method on the truncated normal family.}
    \label{fig: truncated normal family}
\end{figure}

We observed that the spheres in the truncated family on the left of Figure \ref{fig: truncated normal family} are more distorted, compared to the spheres of the non-truncated normal family in Figure \ref{fig: non truncated normal family}. We also saw that, numerically for the truncated pdf $\mathcal N_{[-2,2]}(0,1)$, a radius $\delta$ larger than $0.5$ seems to give a non-compact FR sphere. This surprising phenomenon is discussed in Section \ref{appendix: geodesics and Fisher spheres} of the Appendix in a more general context. 

Furthermore, on the right of both figures, we show the $\delta$-perturbations densities from the nominal density, given a value for $\delta$. They represent distributions that are in the vicinity of the latter for the FR distance. For the truncated pdfs, the $\delta$~- perturbations of $\mathcal N_{[-2,2]}(0,1)$ are quite distinct: some have been affected by the truncation procedure, while not much for others. In addition, some have flattened and are close to the uniform distribution on $[-2,2]$.

\section{Illustration of the density perturbation method and Fisher-Rao spheres} \label{sec: Fisher-Rao spheres illustration}

For the sake of clarity, we illustrate here the nature of $\delta$-perturbations for  some well known parametric families, in usual and truncated situations. These $\delta$-perturbations are obtained by computing the FR spheres $\Lambda_\delta$ in $\mathcal P_i=\{f_{i,\theta}\}_{\theta\in\Theta}$, that constraints the optimization problem $(\star)$. Let us first state the following classical result from Riemannian geometry.
\begin{proposition}[\cite{do1992riemannian}, page 65] \label{prop: geod spheres}
Let $(M,g)$ be a Riemannian manifold. The sphere $\mathcal S(p,\delta)$ for the geodesic distance centered at $p \in M$ with small enough radius $\delta>0$ can be identified as
    $$\mathcal{S}(p,\delta)=\big\{\gamma(1) \ | \ \gamma \text{ geodesic s.t. } \gamma(0)=p, \dot \gamma(0)=v \text{ and } |v|_p=\delta \big\}.$$
     
\end{proposition}
\textcolor{black}{This result is not always true for large radii, see Appendix \ref{appendix: geodesics and Fisher spheres} for more details and Figure \ref{fig: truncated normal family} for an illustration of a possible outcome.} However, Proposition \ref{prop: geod spheres} allows us to discretize small geodesic spheres through the following algorithm, valid in any Riemannian manifold, where $K$ is the number of discretization points 
\texttt{
\begin{enumerate}
    \item pick $K$ tangent vectors $v_1,...,v_K$ at point $p\in M$ with norm $\delta$,
    \item solve the geodesic equation (\ref{geodesic equation}) $K$ times:
$$\ddot \gamma^k + \dot \gamma^i \bar\Gamma_{ij}^k \dot \gamma^j = 0, \ \ \text{for all } k=1,...,\dim(M),$$
 with initial conditions $\gamma_0=p$ and $\dot \gamma_0=v_\ell$, for $\ell\in \{1,...,K\}$, and denote $\bar \gamma_{v_\ell}$ the numerical approximation,
 \item define the discretized sphere as the collection of points $\{\bar \gamma_{v_\ell}(1)\}_{1\leq \ell\leq K}.$
\end{enumerate}
}
Hence for a parametric family $\mathcal P_i=\{f_{i,\theta}\}_{\theta\in\Theta}$, we can apply this algorithm in the coordinate chart $\Theta$ to obtain $K$ points $\theta_\ell\in \Theta$ such that $\{f_{i,\theta_\ell}\}_{\ell=1,...,K}$ discretizes the sphere $\Lambda_\delta$ in $\mathcal P_i$.
In practice, since we mainly consider two-parameter families, the $K$ tangent vectors $v_\ell$ with norm $\delta$ are taken at uniform angle i.e. proportional to the vectors $\big(\cos\left(2\pi \ell/K\right),\sin\left(2\pi \ell/K\right)\big)^\intercal$, $\ell=1,\ldots,K$. 

For each two-parameter family considered below, 
we computed the FIM and the Christoffel symbols $\bar \Gamma_{ij}^k$ (\ref{Christoffel symbols}), and applied Euler method to solve the geodesic equation (\ref{geodesic equation}). This  slightly differs from the Hamiltonian method used in \cite{Gauchy2022}. When the Christoffel symbols are not known explicitly (Gumbel and Gamma), we approximated them using (basic) numerical integration and differentiation methods.

\subsection{Triangular distributions }\label{sec:triangular.dist}
The family $\mathcal T_{[a,b]}$ of triangular distributions supported on $[a,b]$ contains the densities $q_m$ given by 
$$q_m(x)= 2(b-a)^{-1}\left[(x-a)(m-a)^{-1}\textbf{1}_{x\in [a,m]} + (b-x)(b-m)^{-1}\textbf{1}_{x\in (m,b]}\right],\ \ \ m\in [a,b].$$
The Fisher information is given by
$i_m = [(b-m)(m-a)]^{-1}$  for $m \in (a,b)$.
Since $\mathcal T_{[a,b]}$ is a one parameter family, the FR distance $d$ can be easily computed:
$$d(q_{m_0},q_{m_1}) = \left|\int_{m_0}^{m_1} \sqrt{i_m} \ dm\right| = \left| \arcsin\left(\frac{m_1-\alpha}{\sqrt{\beta}}\right) - \arcsin\left(\frac{m_0-\alpha}{\sqrt{\beta}}\right)\right|, $$
where $\alpha:=\frac{a+b}{2}$ and $\beta:=\left(\frac{a-b}{2}\right)^2$. This allows to explicitly compute the FR spheres (see Appendix \ref{appendix: proof of exponential and triangular}).
Note that the radius cannot be too large as $\mathcal T_{[a,b]}$ has finite diameter
$$ \text{diam}(\mathcal T_{[a,b]}) := \sup_{m_0,m_1 \in [a,b]} d(q_{m_0},q_{m_1}) = |\arcsin(-1)-\arcsin(1)| = \pi.$$
Hence, the radius $\delta$ satisfies $\delta\leq \delta_{\max}=\frac{\pi}{2}$. 
Figure \ref{fig:triang perturbation} gives an illustration. 

\begin{figure}[!ht]
     \centering
     \begin{tabular}{ll}
     \begin{minipage}{0.5\textwidth}
         \includegraphics[scale=0.135]{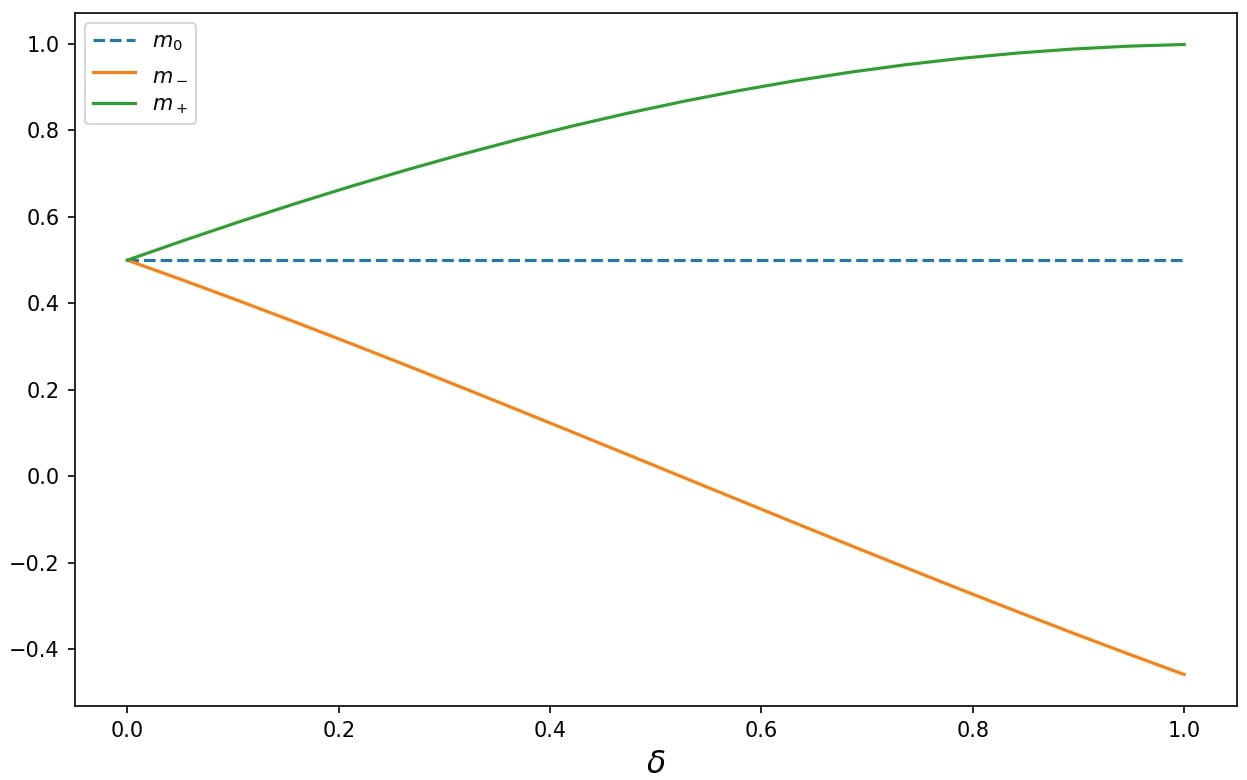}

    \footnotesize Evolution as a function of $\delta$ of points $m_-,m_+$ on the sphere centered at $m_0=0.5$ with radius $\delta$.
    \end{minipage}
     & 
     \begin{minipage}{0.5\textwidth}
         \vspace{0.25cm}
         \hspace{-0.5cm}\includegraphics[scale=0.205]{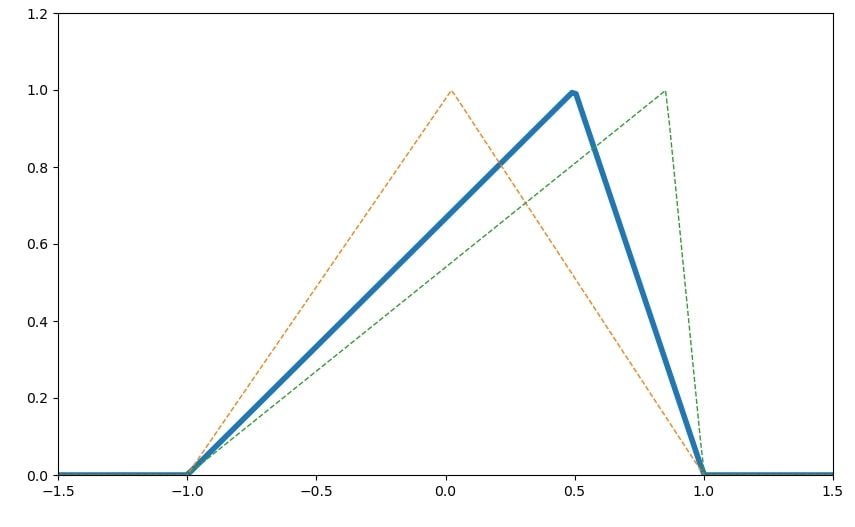}
         
         \footnotesize The $\delta$-perturbations (in dashed lines) of the $q_{0.5}$ density (in solid lines) in $\mathcal T_{[-1,1]}$ for $\delta=0.5$.
         
    \end{minipage} \\
    \end{tabular}
    \caption{The $\delta$-perturbations in the triangular family $\mathcal T_{[-1,1]}$.}
    \label{fig:triang perturbation}
\end{figure}

The family $\mathcal T_{[a,b]}$, being a (connected\footnote{A connected manifold is a set that is not formed by two or more disjoint sets i.e. it is in one piece.}) $1$-dimensional Riemannian manifold, is isometric\footnote{Two spaces are isometric if there exists a smooth invertible map between them that preserves the Riemannian structure.} to an Euclidean (straight) interval, of which there exists three types: a half line $[0,\infty)$, a finite interval $[0,\ell]$ (for some positive $\ell\in \mathbb R$) or the whole line $\mathbb R$. The family $\mathcal T_{[a,b]}$ corresponds to $[0,\pi]$ since it has a diameter of $\pi$.

\subsection{Location-scale families}\label{sec: loc-scale families}
Location-scale families are two parameter families obtained by translating and scaling an initial smooth probability density $p: \mathbb R \to (0,\infty)$. Let $m\in\mathbb R$ and $s>0$, then
$$ p_\theta(x) = \frac{1}{s} p\left(\frac{x-m}{s}\right), \ \ \ \ \theta=(m,s).$$
Here, $m$ and $s$ are resp. the location and scale parameters. 
Many well known families fall in this model. For instance, this is the case for the normal, Gumbel and Cauchy distributions. The following proposition shows that all location-scale families share a common geometry. 

\begin{proposition}\label{prop: loc scale FIM}
    The FIM of a location-scale family is given by
    $I_\theta = \begin{bmatrix}
        \alpha/s^2 & \gamma/s^2 \\
        \gamma/s^2 & \beta/s^2 
    \end{bmatrix},$
    where $\alpha,\beta,\gamma$ do not depend on $\theta=(m,s)$ \textcolor{black}{and are given by integral expressions involving $p$ and $p'.$ }
    In addition, there exists a linear reparameterization with Jacobian matrix $P^\intercal$ such that
    $$I_\theta = P\cdot \left(\frac{1}{s^2} \begin{bmatrix}
        1 & 0 \\
        0 & 1
    \end{bmatrix}\right) \cdot P^\intercal.$$
    Hence, we find that the Fisher metric of location-scale families is nothing more than the Poincaré metric.
\end{proposition}

The proof of Proposition \ref{prop: loc scale FIM} is based on the spectral theorem. 
We also refer to \cite{komaki2007bayesian} for a proof dealing with an even density $p:\mathbb R \to (0,\infty)$. Up to our knowledge, concluding the same for truncated location-scale families is an open problem.

Numerical computations of FR spheres for the Gumbel family in the truncated ($\text{Gumb}_{[a,b]}$) and non-truncated ($\text{Gumb}$) situations, as well as the $\delta$-perturbations for both, are illustrated in Figure \ref{fig:Truncated and non truncated gumbel spheres and perturbations}.

\begin{figure}[!ht]
     \centering
     \begin{tabular}{ll}
     \begin{minipage}{0.5\textwidth}
         \includegraphics[width=1\textwidth]{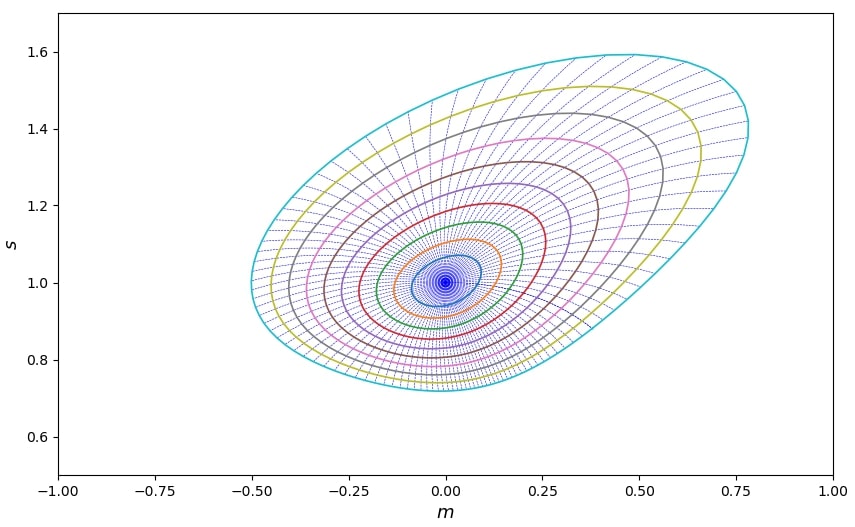}
       
         {\footnotesize Non-truncated case: concentric FR spheres centered at $\text{Gumb}(0,1)$ with maximal radius $\delta=0.5$ in the $(m,s)$ parameter space.}
             \label{fig:gumbel sphere non tr}
    \end{minipage}
    & 
     \begin{minipage}{0.5\textwidth}
         \includegraphics[width=1\textwidth]{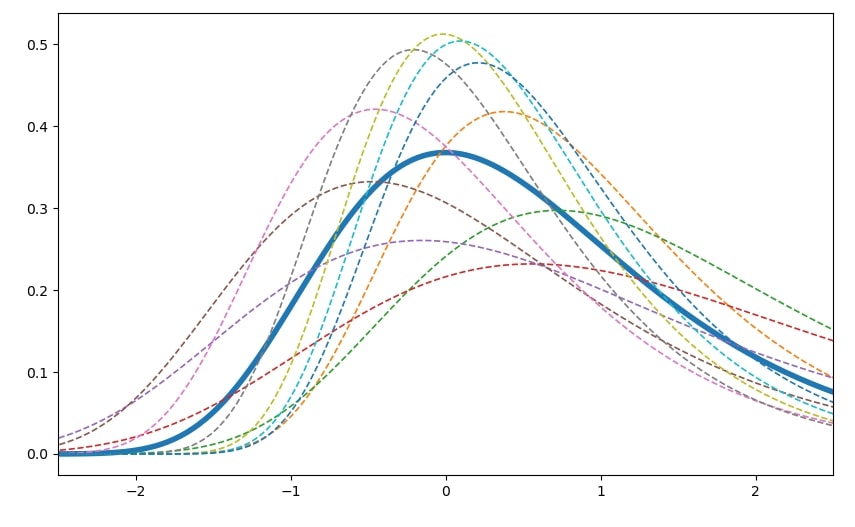}

         {\footnotesize Non-truncated case: some $\delta$-perturbations (in dashed lines) of the $\text{Gumb}(0,1)$ density for $\delta = 0.5$.}
         \label{fig:graph gumb (0,1) densities}
    \end{minipage} \\
    & \\
     \begin{minipage}{0.5\textwidth}
         \includegraphics[width=1\textwidth]{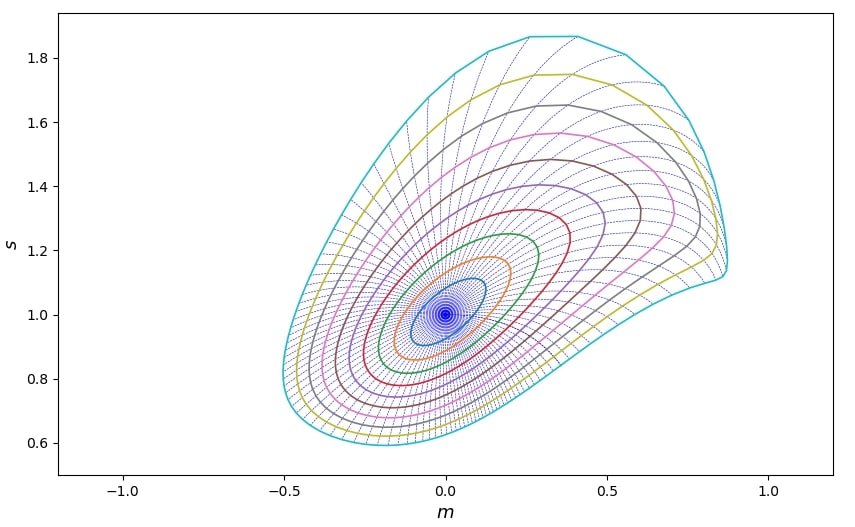}

         {\footnotesize Truncated case: concentric FR spheres centered at $\text{Gumb}_{[-2,2]}(0,1)$ with maximal radius $\delta=0.5$ in the $(m,s)$ parameter space.}
             \label{fig:gumbel sphere tr}
    \end{minipage}
     & 
     \begin{minipage}{0.5\textwidth}
         \includegraphics[width=1\textwidth]{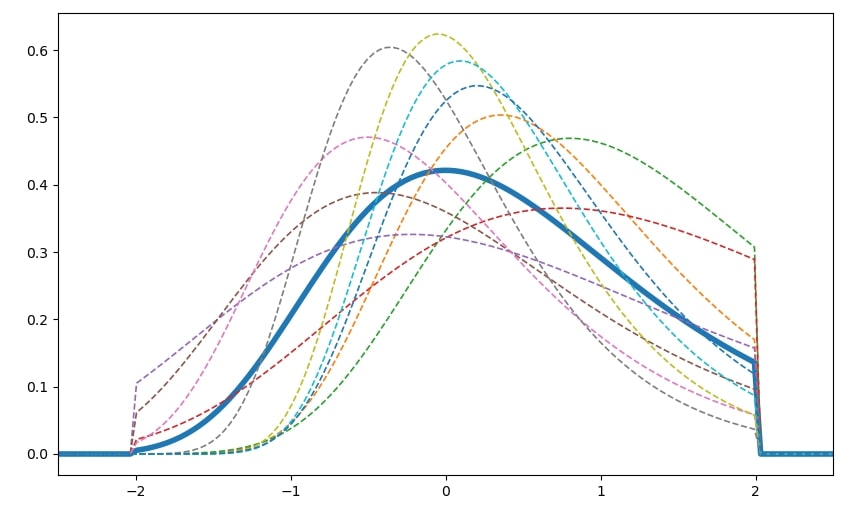}

         {\footnotesize Truncated case: some $\delta$-perturbations (in dashed lines) of the truncated $\text{Gumb}_{[-2,2]}(0,1)$ density for $\delta=0.5$.}
         \label{fig:graph tr gumb (0,1) densities}
    \end{minipage}
    \end{tabular}
    \caption{FR spheres and $\delta$-perturbations in the usual and truncated case (on $[-2,2]$) for the Gumbel family. 
    The concentric spheres in the upper left figure look like a linear transformation of the concentric spheres in the usual normal family from Figure \ref{fig: non truncated normal family}, thus illustrating the claim of Proposition \ref{prop: loc scale FIM}.}
        \label{fig:Truncated and non truncated gumbel spheres and perturbations}
\end{figure}

\subsection{Bijective push-forward of a parametric family}

Here, we work with a family of probability distributions obtained by pushing forward a reference family $\mathcal P :=\{P_\theta\}_{\theta \in \Theta}$ on some space $\mathcal X$ to $\mathcal Y$ by a smooth invertible map $$h:\mathcal X \to \mathcal Y.$$ 
For instance, if we take $h(x) = e^x$ and $\mathcal P$ to be the family of normal distributions, then $h_\#\mathcal P := \{h_\# P_\theta\}_{ \theta\in \Theta}$ becomes the log-normal family. Here $h_\# P_\theta$ denotes
the push-forward distribution. Set 
$$R_\theta := h_\# P_\theta$$
and let $\mathcal R:= \{R_\theta\}_{\theta \in \Theta}$ be the push-forward family. Our goal is to determine the FIM of a truncated version of $\mathcal R$ on some subset of $\mathcal Y$. Let us first remind the following theorem
\begin{theorem}[\cite{calin2014geometric}, page 25] If $\mathcal X,\mathcal Y$ are open subsets of $\mathbb R^d$ and $h:\mathcal X \to \mathcal Y$ is smooth and invertible, then for any regular parametric family $\mathcal P$ on $\mathcal X$ the pushforward family $\mathcal R:=h_\#\mathcal P$ has the same FIM as $\mathcal P$.
\end{theorem}

This result directly implies that geodesics and FR spheres in $\mathcal P$ are the same in $\mathcal R$ (i.e., they are isometric). As an example, this holds for the normal and the log-normal families.

Now, we would like to compute the FIM of a truncated version of $\mathcal R$. Let us first fix some notations. From now on, we denote $B\subset \mathcal Y$ the truncation domain (a measurable subset) for $\mathcal R$ and  $A:= h^{-1}(B)\subset \mathcal X$ its preimage. 
We denote the truncated family of $\mathcal P$ on $A$ (resp. of $\mathcal R$ on $B$) as $\mathcal P_A :=\{P_\theta^A\}_{\theta \in \Theta}$
(resp. $\mathcal R_B :=\{R_\theta^B\}_{\theta \in \Theta}$). 
In addition, let $I_\theta^A$ (resp. $K_\theta^B$) denotes the FIM of $\mathcal P_A$ (resp. of $\mathcal R_B$).

Since we want to compute the FIM of $\mathcal R_B=\{(h_\#P_\theta)^B\}_{\theta \in \Theta}$, if we apply the previous theorem on $\mathcal P_A=\{P_\theta^A\}_{\theta \in \Theta}$ with the restriction of $h$ to $A$, we get that $\mathcal P_A$ and 
\begin{align*}\label{eq: pushforward P by h}
h_\#\mathcal{P}_A := \left\{h_\#\left(P_\theta^A\right)\right\}_{\theta \in \Theta}
\end{align*}
have the same FIM. 

The following result states that $h_\#\mathcal P_A$ and $\mathcal R_B$ are the same, %
the proof being obvious and intuitive. 
\begin{lemma}\label{lemma: push forward}
Under the previous notations, it comes  $$(h_\#P_\theta)^B = h_\#(P_\theta^A), \ \ \ \ \text{for all $\theta$ in $\Theta$.}$$
Therefore $\mathcal R_B$ and $h_\#\mathcal P_A$ are the same family.
Moreover, $\mathcal R_B$ and $\mathcal P_A$ have the same Riemannian structure since their FIM coincide:
$$I_\theta^A = K_\theta^B, \ \ \ \ \text{for all $\theta$ in $\Theta$.}$$
\end{lemma}

This result allows us to deduce $K_\theta^B$, the FIM of the truncated push-forward family $\mathcal R_B$ on $\mathcal Y$, if we know $I_\theta^A$, the FIM of the initial truncated family $\mathcal P_A$ on $\mathcal X$ where $A =h^{-1}(B)$, see Example \ref{ex: log-normal FIM from normal}. The following diagram summarizes our considerations:
\begin{center}
\begin{tikzcd}
(\mathcal X, P_\theta) \arrow[r, "\text{tr. on $A$}"] \arrow[d,  "\ h_\#"]    & (\mathcal X, P_\theta^A) \arrow[d, "\ h_\#"] \\
(\mathcal Y, R_\theta) \arrow[r, "\text{tr. on $B$}" ]    &  (\mathcal Y,R_\theta^B)
\end{tikzcd}
\end{center}
The previous lemma tells that the push-forward operation and the truncation operation commute. In addition, the standard push-forward theorem shows that the left-down arrow preserves the Fisher information, and the lemma guarantees that the Fisher information is preserved by the right-down arrow.

\begin{example} \label{ex: log-normal FIM from normal}
    For the family $\log \mathcal N$ of log-normal distributions
    $$r_\theta(x) =\frac{1}{x\sigma\sqrt{2\pi}} \exp\left(-\frac{(\log x -\mu)^2}{2\sigma^2}\right)\textbf{1}_{x>0},  \ \ \text{ for } \theta = (\mu,\sigma),$$
    Lemma \ref{lemma: push forward} implies that the FIM of $\log\mathcal N_{[a,b]}$, the truncated version of this family on $[a,b]$, is given by
    $I_\theta ^{[\log a, \log b]}$,
    where $I^{[\alpha,\beta]}$ is the FIM of $\mathcal N_{[\alpha,\beta]}$ provided in Proposition \ref{prop: tr normal FIM}. In particular, the Levi-Civita Christoffel symbols of $\mathcal N_{[\log a,\log b]}$ and $\log\mathcal N_{[a,b]}$ are identical 
    as well as their geodesics and FR spheres on the upper halfplane. Figure \ref{fig: non truncated and truncated lognormal perturbed densities} gives an illustration of this result.
\end{example}

\begin{figure}[!ht]
     \centering
     \begin{tabular}{ll}
     \begin{minipage}{0.5\textwidth}
         \hspace{-10mm}\includegraphics[scale=0.25]{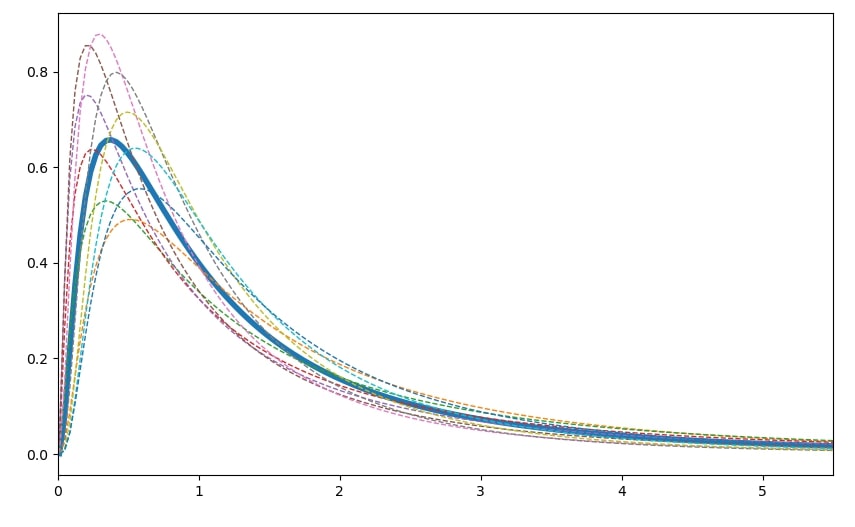}

    \footnotesize Some $\delta$-perturbations of $\log \mathcal N(0,1)$ for $\delta = 0.3$.
    \label{fig: non truncated lognormal perturbed densities}
    \end{minipage}
     & 
     \begin{minipage}{0.5\textwidth}
         \vspace{0.25cm}
         \hspace{-0.5cm}\includegraphics[scale=0.25]{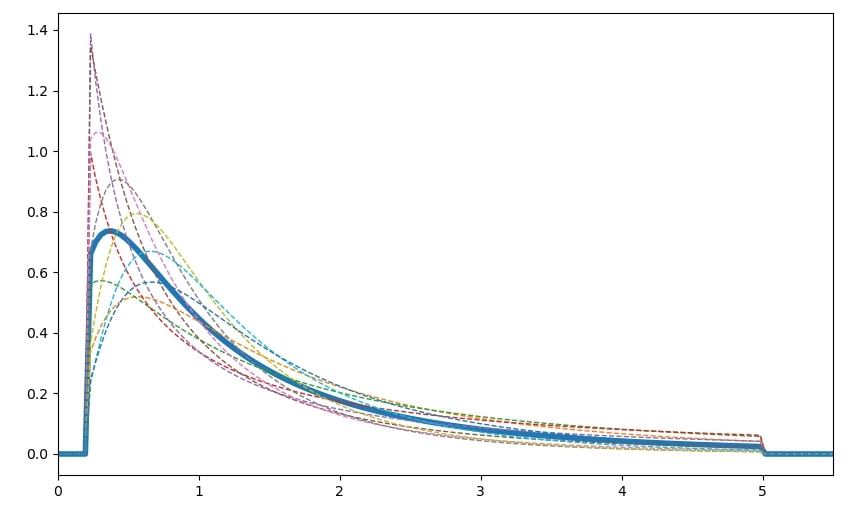}

         \footnotesize Some $\delta$-perturbations of the truncated log-normal $\log \mathcal N_{[0.2,5]}(0,1)$ for $\delta = 0.3$.
         \label{fig:truncated lognormal perturbed densities}
    \end{minipage} \\
    \end{tabular}
    \caption{$\delta$-perturbations of a log-normal density (left: usual, right: truncated).}
    \label{fig: non truncated and truncated lognormal perturbed densities}
\end{figure}

A natural question would be to know whether the previous lemma still applies when $h:\mathcal X \to \mathcal Y$ is a sufficient statistic for the family $\mathcal P$. Indeed, in this case, the Fisher metric of $\mathcal P_A$ and $\mathcal R_B$ are the same. The proof, which is obvious, is provided in Appendix \ref{appendix: proof of push-forward}.

\subsection{Exponential families}

Exponential families are wide families of probability distributions, with normal, log-normal, gamma, binomial, Poisson and geometric distributions amongst their most notable members. The densities $f_\theta$ of an exponential family with reference distribution $\nu$ on $\mathbb R^m$ are given, for $\theta \in \Theta$, by
\begin{align*}
f_\theta(x)=e^{x\cdot\theta - \psi(\theta)}, \ \ \ x\in \mathbb R^m, \theta \in \Theta. 
\end{align*}
Here, we denote $x\cdot \theta$ the Euclidean scalar product on $\mathbb R^m$, the function $\psi$ is the log-Laplace transform of $\nu$ and $\Theta$ is the domain of $\psi$. The FIM for exponential families is given by the Hessian of $\psi$ \cite{calin2014geometric}, hence exponential families can be seen as Hessian manifolds \cite{shima2007geometry}.

A truncated exponential family is built from an initial exponential family truncated on an interval $[a,b]$, here we take $m=1$. These truncation bounds are usually considered as parameters \cite{akahira2016second}, and therefore these families are not regular and the FIM is not necessarily well defined. In this subsection, we consider truncated exponential families on an interval $[a,b]$ where $a,b$ are fixed  -- a special case of truncated distributions studied by \cite{yoshioka2023information}. 

The truncated densities with normalizing constant $N_\theta$ also form an exponential family since
$$N_\theta^{-1}f_\theta(x)\textbf{1}_{x\in [a,b]}d\nu(x) = e^{x\cdot \theta - \psi(\theta) - \log N_\theta}d\nu_{[a,b]}(x),$$
where $\nu_{[a,b]}(\cdot):=\int_{\cdot} \textbf{1}_{[a,b]} d\nu$ is the reference distribution and $\tilde \psi(\theta):= \psi(\theta) + \log N_\theta$ is the log-Laplace transform of $\nu_{[a,b]}$. Hence the FIM of this family is given by the Hessian of $\tilde \psi$. 
See Figure \ref{fig: density perturb exp families} for illustrations in the non-truncated case, and refer to Appendix \ref{appendix: proof of exponential and triangular} and the Supplementary Material for details on computations.

\begin{figure}[!ht]
     \centering
     \begin{tabular}{ll}
     \begin{minipage}{0.5\textwidth}
         \includegraphics[width=1\textwidth]{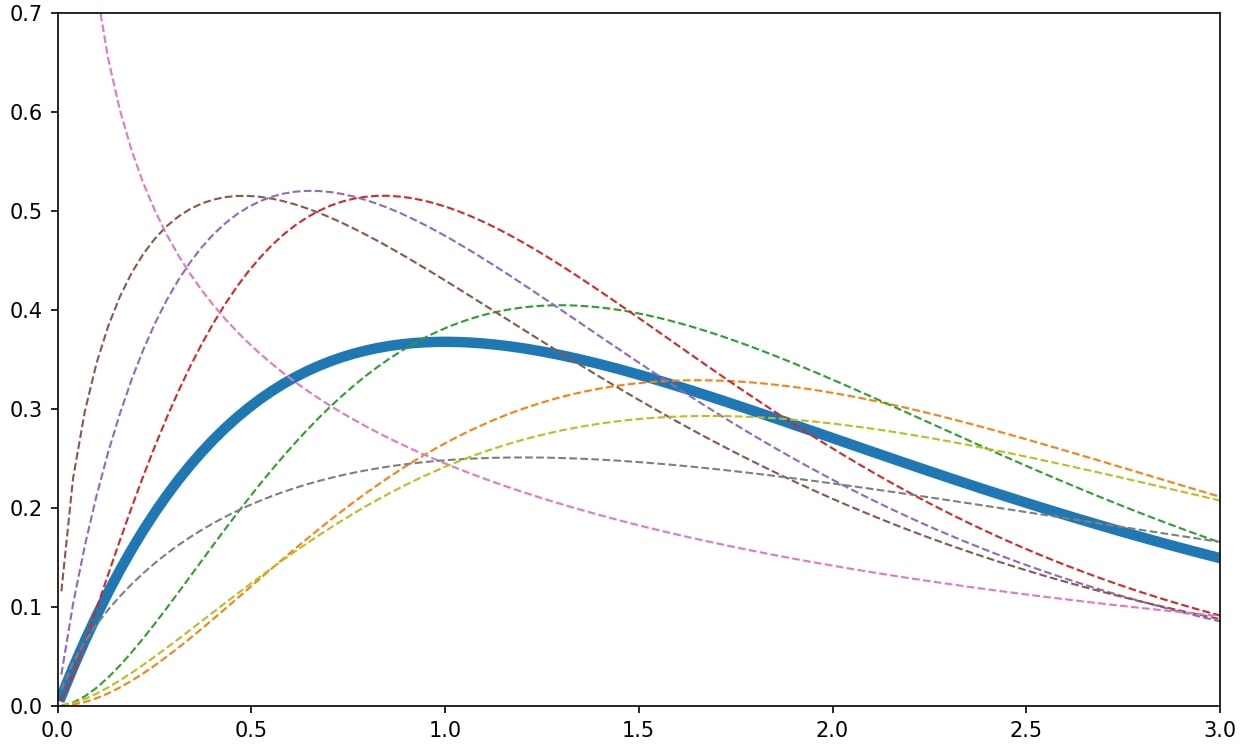}

         {\footnotesize Some $\delta$-perturbations of the $\Gamma(2,1)$ Gamma pdf, with $\delta=0.5$. }
             \label{fig:gamme density pertrub}
    \end{minipage}
     & 
     \begin{minipage}{0.5\textwidth}
         \includegraphics[width=1\textwidth]{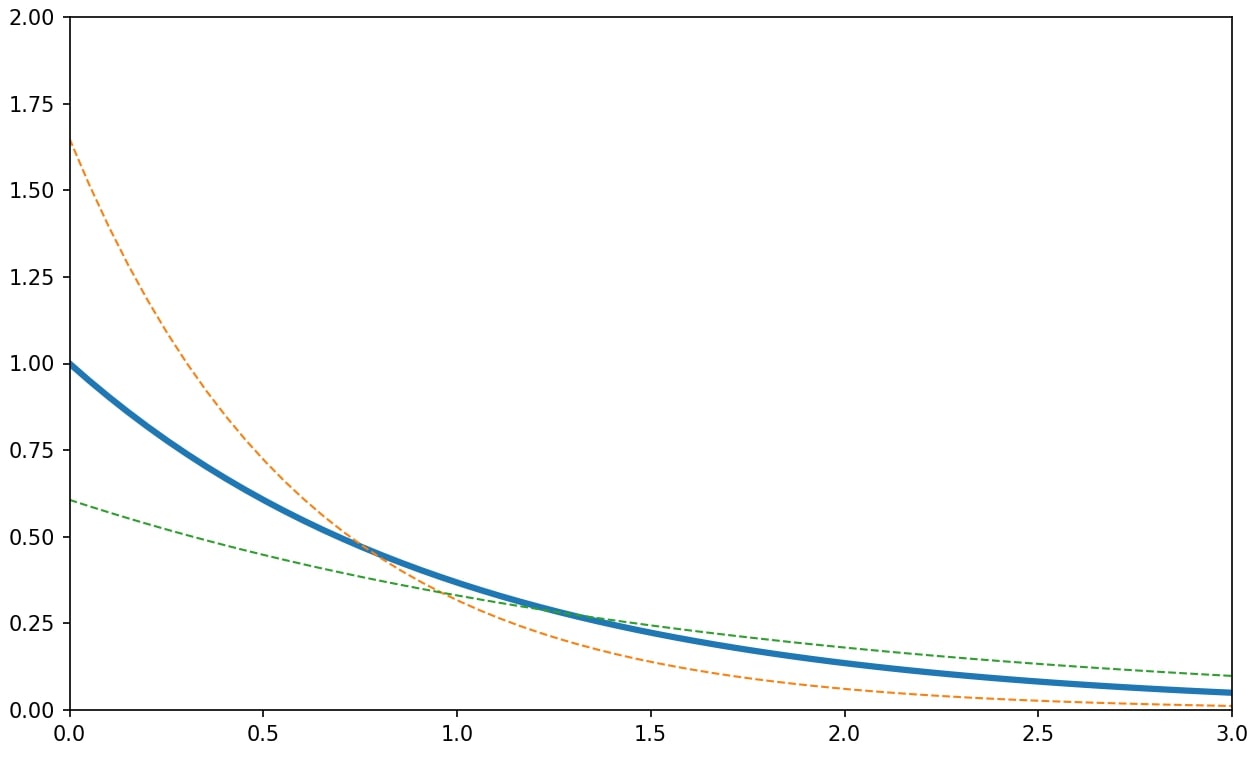}

         {\footnotesize  The $\delta$-perturbations of an exponential pdf of parameter $\lambda~=~1$, with $\delta=0.5$.}
         \label{fig:exponential distrib density perturb}
    \end{minipage}
    \end{tabular}
    \caption{Some density perturbations for two members of exponential families.}
    \label{fig: density perturb exp families}
\end{figure}

\section{Numerical experiments: robustness study of a flood model}\label{sec:numerical experiment}

In this section, the RA methodology discussed in Section \ref{sec:RA framework and truncated distributions} to is applied for studying the robustness of an analytic flood model. This model has been studied in previous articles such as \cite{Lemaitre2015, Gauchy2022}. The code we developed is available at \href{https://github.com/baalub/Truncated-distributions.git}{https://github.com/baalub/Truncated-distributions.git}

\subsection{Context}
The flood model computes a river's maximal annual water level $H$ assuming constant and uniform flow rate. This is useful information for predicting if flooding will occur. The model is defined as 
$$ H := G(K,Q,Z_m,Z_v) = Q^{0.6}\left(300K \sqrt{\frac{Z_m - Z_v}{5000}} \right)^{-0.6},$$
where $Q\in [0,3000]$ is the maximal annual flow rate, $K \in [15,90]$ is a roughness (Strickler-Manning) coefficient and $Z_m \in [54,56]$ and $Z_v \in [49,51]$ are resp. the upstream and downstream heights of the river. Figure \ref{fig:modele crue} gives an illustration of the river and Table \ref{tab: input distributions} summarizes the baseline distributions, selected from legacy data and  expert knowledge.

\begin{figure}[ht]
    \centering
    \includegraphics[width=0.4\textwidth]{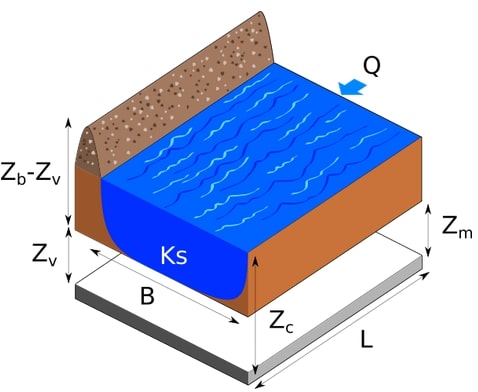}
    \caption{Parameters involved in the analytic flood risk model. Courtesy of Merlin Keller (EDF R\&D).} 
    \label{fig:modele crue}     
\end{figure}

\begin{table}[h!]
\label{table 1}
\centering
\begin{tabular}{ p{1cm}p{7cm}p{3.5cm} }
 \hline
 Inputs & Parametric family & Baseline distribution \\
 \hline
 $Q$  &   Truncated Gumbel 
 $\text{Gumb}_{[0,3000]}$ & $m=1013,\ s=558$ \\[0.5ex]
 $K$  &   Truncated Normal $\mathcal N_{[15,90]}$ & $\mu=30,\ \sigma=7.5$   \\[0.5ex]
 $Z_m$  & Triangular $\mathcal T_{[54,56]}$ & $m =55$  \\[0.5ex]
 $Z_v$  & Triangular $\mathcal T_{[49,51]}$ & $m=50$  \\[0.5ex]
 \hline
\end{tabular}
\caption{Table of baseline input distributions for the analytic flood model.}
\label{tab: input distributions}
\end{table}

The output $H$ corresponds to the river depth and flooding will occur if $H + Z_v\geq Z_d$, where $Z_d$ is the height of the dyke. The QoI we look at is the $90\%$ quantile of the output $H$.
We perform the RA method presented above to determine which inputs are robust w.r.t. input density perturbations. As indicated in Section \ref{sec:RA.framework}, we compute 
\begin{align*}
    \min_{f_{i\delta} \in \Lambda_{i\delta}} \widehat S_i(f_{i\delta}) \ \ \ \ \text{and} \ \ \ \max_{f_{i\delta} \in \Lambda_{i\delta}} \widehat S_{i}(f_{i\delta}) \tag{$\star \star$}
\end{align*}
where $\widehat S_{i}$ is an estimator of $S_{i}$ (see Appendix \ref{appendix: estimation procedure} for details), $i\in\{Q,K,Z_m,Z_v\}$ and $\delta \in \{0.1,0.2,...,1\}$. Here, we picked $\delta_{\max}=1$ for practical reasons. Three-dimensional plots of the quantile estimator on the concentric FR spheres are provided on Figure \ref{fig: graph quant estim}. Besides, Figure \ref{fig:delta plot PLI} highlights the maximum and minimum values of $\widehat S_{i}$ on concentric FR spheres, as a function of $\delta$. The confidence intervals accompanying this result are built using bootstrap techniques, detailed in  Appendix \ref{appendix: estimation procedure}. \\

\begin{figure}[ht]
     \centering
     \begin{tabular}{ll}
     \begin{minipage}{0.4\textwidth}
         \includegraphics[width=1\textwidth]{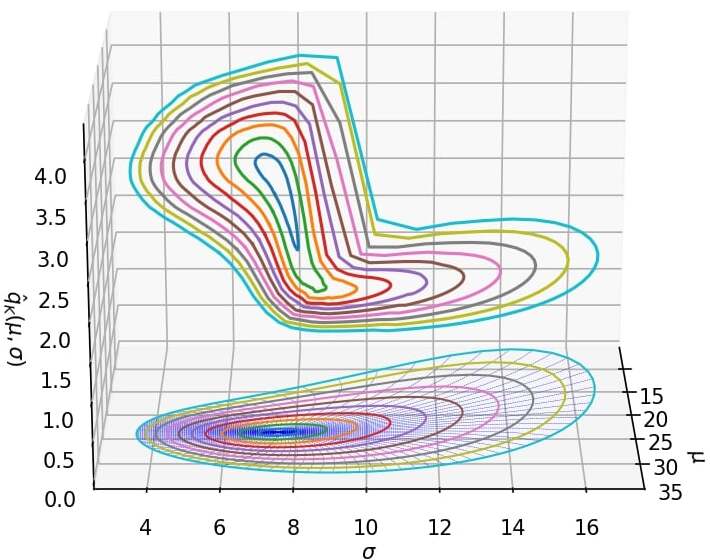}

         {\footnotesize  Graph of $(\mu,\sigma)\mapsto \widehat q^\alpha$ with perturbation of input $K$ distribution (truncated Gaussian).}
    \end{minipage}
     & 
     \begin{minipage}{0.4\textwidth}
         \includegraphics[width=1\textwidth]{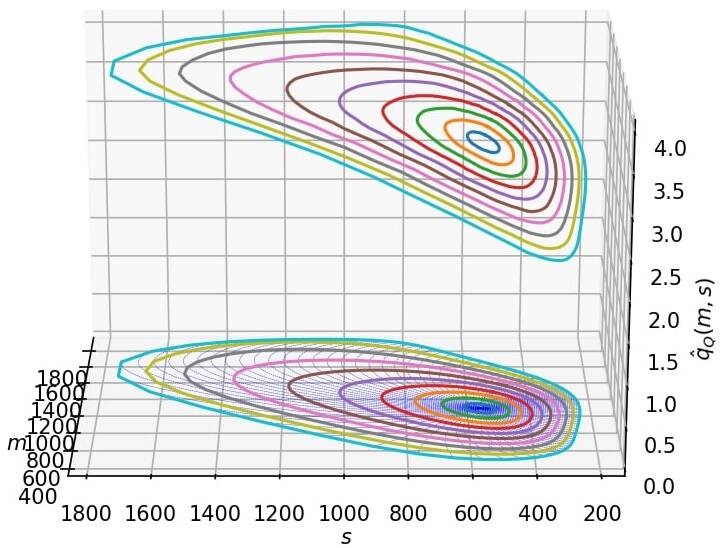}

         {\footnotesize Graph of $(m,s)\mapsto \widehat q^\alpha$ with perturbation of input $Q$ distribution (truncated Gumbel).}
    \end{minipage}
    \end{tabular}
    
    \caption{Three-dimensional plots of the quantile estimator for two inputs of the flood model. The lower plane is the parameter space for the respective inputs on which are shown the concentric FR spheres. The value of the quantile estimator $\widehat q^\alpha$ is then plotted on each of these concentric spheres.}
    \label{fig: graph quant estim}
\end{figure}

\subsection{Optimization procedure}
Computing the bounds in $(\star \star)$ for the triangular inputs $Z_m$ and $Z_v$ was not challenging since in these cases the FR spheres $\Lambda_{i\delta}$ were made up of two points. For the truncated normal $K$ and truncated Gumbel $Q$ inputs, our approach for computing $(\star \star)$ is to discretize each sphere $\Lambda_{i\delta}$, for $\delta \in \{0.1,0.2,...,1\}$, into $100$ points. We then optimize the estimator $\widehat S_{i}$ of the robustness index. We use this approach because the estimator $\widehat S_{i}$ is piecewise constant (see Appendix \ref{appendix: estimation procedure}) and thus a gradient-based method cannot be used. Improving this optimization procedure is a subject of current work, evoked further in the Discussion section.  

Discretizing the concentric FR spheres takes around $8$ minutes for each input $K$ and $Q$. Evaluating $\widehat S_{i}$ on each point of the discretization ($1000$ points) for a sample of size $N=10^4$ took around $3$ hours in total for all the inputs $K,Q,Z_m,Z_v$ on \textcolor{black}{an $11^{th}$ gen Intel(R) Core i5-1135G7 @2.40Hz computer}. The time required for the estimation of the robustness index highly depends on the choice of the QoI. Usually, estimating the expectation or a probability threshold of the output takes significantly less time than a quantile or a superquantile.

\subsection{Analysis of the numerical results}

From Figure \ref{fig:delta plot PLI}, we observe that the output of the model $G$ is sensitive to density perturbations on $K$ and $Q$ of level of at most $\delta_{\max} = 1$. For the inputs $Z_m$ and $Z_v$, it seems that perturbing their initial densities does not affect much the output.

Figure \ref{fig:evolution argmin Gumbel delta} highlights that as the perturbation level $\delta$ increases a minimizer of $\widehat S_Q$ has a tendency to give more weight to points in $[0,3000]$ that are close to $0$. 
This is not surprising since $G$ is increasing as a function of $Q$.

\begin{figure}[!ht]
     \centering
        \includegraphics[width=0.7\textwidth]{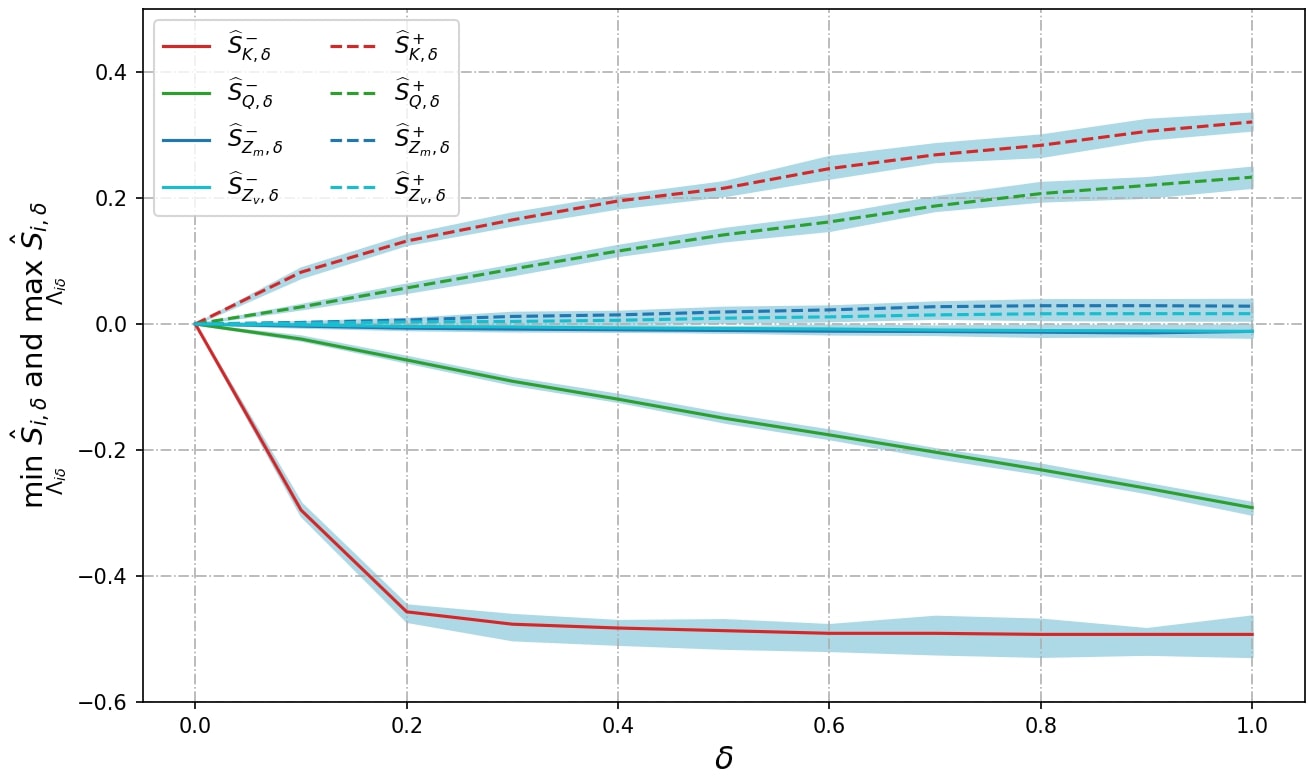}
         \caption{Evolution of the maximum and minimum value on $\Lambda_{i\delta}$ of $\widehat S_{i\delta}$ as a function of $\delta$, \textcolor{black}{where the 80\% bootstrap confidence intervals are shown in light blue.}}
         \label{fig:delta plot PLI}
\end{figure}

\begin{figure}[!ht]
         \centering
         \includegraphics[width=0.6\textwidth]{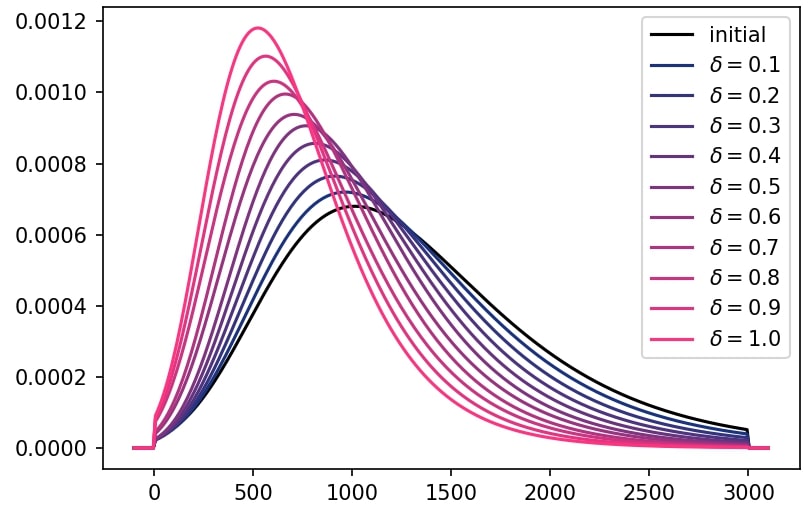}
         \caption{Evolution as a function of $\delta$ of an \textcolor{black}{argmin} for $(\star\star)$ for the Gumbel input $Q$.}
         \label{fig:evolution argmin Gumbel delta}
\end{figure}

\textcolor{black}{This RA result for the flood model is comparable to previous RA studies in papers such as \cite{Gauchy2022}. Indeed, in this paper, the inputs $Z_m$ and $Z_v$ were also the most robust to input density perturbations. However, in contrast to our case, the input $Q$ was the least robust in \cite{Gauchy2022}. This may be due to the difference in the baseline distributions for $Q$ and $K$.}

\section{Discussion} \label{sec: discussion}

In this paper, we have explained the RA methodology based on the FR distance on families of probability distributions. This distance allowed us to define $\delta$~-~perturbations of densities in some parametric family. We then illustrated the feasibility of this $\delta$~-~perturbation method in some classical families of distributions both in the truncated and non-truncated case. This required either an explicit formula for the FR distance or an approximation of the Christoffel symbols of the corresponding Fisher metric. This showed that the FR distance, which might seem hard to compute at first glance, can be handled rather easily and that the FR spheres can be approximated with good precision (at least for small radii). We also proved a result that allows us to deduce the Fisher metric for some truncated families obtained by a bijective push-forward operation of another truncated family (Lemma \ref{lemma: push forward}). Finally, we applied the RA method to an analytic flood model to illustrate its applicability.

Compared to the previous RA methods, the advantage offered by this approach is that it gives a mathematically rigorous and objective way for defining density perturbations. Nevertheless, the method leads to obstacles that have to be overcome in practice.

First, the optimization method used for problem $(\star \star)$ requires to discretize the FR spheres into finite points and then to optimize $\widehat S_{i}$ on these points. 
The choice for such an optimization procedure is related to the estimation method for $\widehat S_i$, but obviously it does not scale well to higher dimensions. One solution would be to regularize $\widehat S_{i}$ and perform Riemannian gradient descent on the FR spheres. Moreover, one can also consider the same optimization problem constrained to the FR ball with radius $\delta_{\max}$ instead of on each concentric spheres, which reduces to solving a single optimization problem and allows to obtain the maximum and minimum value of $\widehat S_i$ for all $\delta\leq \delta_{\max}$.

Second, we only considered the case where the QoI on the output $Y$ is a quantile for the flood model. However, QoIs for which $(\star)$ becomes an expectation optimization problem can be directly tackled with stochastic optimization algorithms, which allows to avoid optimizing an estimator of the objective function. This will be the topic of a future work.

Another central question is obviously to set a maximum value $\delta_{\max}$ in practice. To do so, let us first note that this maximum value can be imposed by the intrinsic geometry of the family of distributions (see the example of the family of triangular distributions). In contrast, other families with infinite diameter (location-scale families for instance) do not impose such constraints and spheres of arbitrarily large radii are well defined. Furthermore, some families with infinite diameter can contain spheres that are not compact. This may indicate the existence of distributions outside the family that are at finite FR distance. Thus, this can be useful for defining a possible $\delta_{\max}$ at a density $f_{\theta_0}$
$$ \delta_{\max}:= \sup\big\{\delta>0 \ | \ \mathcal S(f_{\theta_0},\delta) \text{ is compact}\big\},$$
where $\mathcal S(f_{\theta_0},\delta)$ is the FR sphere at $f_{\theta_0}$ with radius $\delta$. 

This non-compactness property of spheres is linked to geodesic completeness of the family, through the Hopf-Rinow theorem \cite{do1992riemannian}. 
Up to our knowledge, completeness of families of distributions is not well studied in the literature (even less for truncated ones). Hence determining an intrinsic $\delta_{\max}$ in this way for a specific family is a challenging open problem. Actually $\delta_{\max}$ is only defined for a point in a specific family. Note further that compactness of FR spheres $\Lambda$, coupled with a continuity assumption on $f_{i\delta}\in\Lambda \mapsto S_i(f_{i\delta})$, is useful for guarantying existence of solutions to ($\star$). 

Moreover, choosing a $\delta_{\max}$ should also ensure that distributions within the FR sphere defined in this way have ``practical" plausibility with regard to the UQ problem. For example, we could try to impose that $f_{i\delta_{\max}}$ should correspond to the least informative plausible distribution (e.g., a uniform distribution) or the most penalizing possible distribution (in the sense of least favorable prior distributions in the Bayesian framework \cite{dytso2018structure}).

Another approach, extrinsic to the geometric problem and which would enable us to compare different perturbations approaches for RA, could be to interpret $\delta_{\max}$ as the result of an inverse problem. With a view to testing robustness against a limit risk defined, for example, as a limit output $\alpha-$quantile $q_{\lim}(\alpha)$ for $Y$, with $0<\alpha \ll 1$, as in \cite{benoumechiara2020detecting}, consider
\begin{eqnarray*}
\delta_{i,\max} & = & \arg\min\limits_{\delta_i} \left|q_{\lim}(\alpha) - q(\alpha ; \delta_i) \right|
\end{eqnarray*}
where $q(\alpha ; \delta_i)$ is the $\alpha-$quantile of $Y_i=G(X_1,\ldots,X_{i,\delta_i},\ldots,X_d)$ where the $X_j\sim f_j$ for $j\neq i$ and $X_{i,\delta_i}\sim f_{i\delta_i}$. Such a definition  would give a practical, model-agnostic and more easily interpretable meaning to the different values of $\delta_{\max}$. It would, however, require the resolution of $d$ inverse problems. This could be computationally time-consuming. 

While the methodology has been illustrated in situations where the components of $\bX$ are independent, it is easily adaptable to cases of stochastic dependence modeled by copulas. These are amongst usual implementation choices in UQ studies (e.g., \cite{Baudin2017,benoumechiara2020detecting,torre2019general}). Nevertheless, adaptations are needed to handle other common types of joint modeling, such as hierarchical or graph-based approaches.

\backmatter
\begin{appendices}

\section{Estimation procedure for a given QoI} 
\label{appendix: statistics} 
\subsection{Quantile estimation procedure}\label{appendix: estimation procedure}

For all forward models $G$, a QoI of the output $Y$ often cannot be computed explicitly. Nevertheless, an estimate can be provided by using a statistical estimation procedure. We describe here how we construct this estimator, following  the importance sampling approach suggested by \cite{Gauchy2022,iooss2022bepu}. For better reading and to fit our experiments, we consider that all inputs $\textbf{X}=(X_1,...,X_d)$
are independent, but generalizing the approach to non-independent inputs is straightforward. 
\texttt{
\begin{enumerate}
\item  Sample $\textbf{X}^1,...,\textbf{X}^N $ from the baseline distribution $P_{\bX}$ and denote $Y_j := G(\textbf{X}^j)$. 
\item For each $i=1,\ldots,d$, denote $Y^{i\delta}:=G(\textbf{X}^{i\delta})$ where $\textbf{X}^{i\delta}\sim f_1\otimes \cdot \cdot \cdot \otimes f_{i\delta} \otimes \cdot \cdot \cdot f_d$ and $f_{i\delta}$ is a $\delta$~-~perturbation of $f_i$. 
\item Estimate the cumulative distribution function $F_{Y^{i\delta}}$ of $Y^{i\delta}$ by 
$$\widehat F_{i\delta} (t) = \frac{1}{\sum_{j=1}^N L_{i,j}(\theta)} \sum_{j=1}^N L_{i,j}(\theta)\textbf{1}_{Y_j\leq t},$$
where $\theta \in \Lambda_{i\delta}\subset \Theta_i$ and $L_{i,j}(\theta) ={f_{i\delta}(\textbf{X}^j)}/{f_{i}(\textbf{X}^j)}$ are the likelihood ratios.
\end{enumerate}
}
By the law of large numbers, the estimator $\widehat F_{i\delta} (t)$ converges almost surely to the cumulative distribution function of $Y^\theta$ for all $t\in \mathbb R$. The quantile estimator $\hat q_{i\delta}$ of $q^\alpha(Y^{i\delta})$ is built by plug-in:
$$ \hat q_{i\delta} := \inf \{ t \in \mathbb R \ | \ \widehat F_{i\delta} (t) \geq \alpha \}.$$
Hence $q^\alpha(Y^{i\delta})$ can be estimated for all $i$ and for all $\delta$ given only a sample of $P_{\textbf{X}}$. The estimator of the PLI robustness index $S_{i}$ (Equation (\ref{eq:PLI})) is then given by
$$ \widehat S_{i}(f_{i\delta}) := \frac{\hat q_{i\delta}}{\hat q_i} - 1,$$
where $\hat q_i$ is the empirical quantile estimator of $q^\alpha(Y)$ where $Y:=G(X)$ and $X\sim f_1\otimes\cdot\cdot\cdot\otimes f_d$. The asymptotic properties of $\widehat S_{i}$ are studied in \cite{Gauchy2022}.

The estimators $\widehat F_{i\delta}(t)$ and $\hat q_{i\delta}$ can be seen as functions of the parameter $\theta \in \Theta_i$. The space $\Theta_i$ is a parameter space of the family $\mathcal P_i$ of possible perturbation densities of $f_i$. Since the estimator $\widehat F_{i\delta}(t)$ is piecewise constant as a function of $t$, jumping on the sample points $Y_j$, and as a function of $\theta$, its generalized inverse $\hat q_{i\delta}$ is also piecewise constant taking values in $\{Y_1,...,Y_N\}$ 
$$ \theta \in \Theta_i \to \hat q_{i\delta} \in \{Y_1,...,Y_N\}.$$
Similarly since $\hat q_i$ also takes finite values, this shows that the PLI estimator $\widehat S_{i}$ also takes finite values and 
$$    \min_{f_{i\delta} \in \Lambda_{i\delta}} \widehat S_{i}(f_{i\delta}) \ \ \ \ \text{and} \ \ \ \max_{f_{i\delta} \in \Lambda_{i\delta}} \widehat S_{i}(f_{i\delta})$$
is a discrete optimization problem. Therefore, a gradient-based optimization method cannot be used directly.

\subsection{Bootstrap confidence intervals}\label{appendix: bootstrap confidence interval}
We used a bootstrap procedure to compute confidence intervals for
$$ \min_{f_{i\delta} \in \Lambda_{i\delta}} S_{i}(f_{i\delta}) \ \ \ \ \text{and} \ \ \ \max_{f_{i\delta} \in \Lambda_{i\delta}} S_{i}(f_{i\delta}).$$
Focusing only on the maximization problem, for each $\delta$, we approximated the maximum value of $\widehat S_{i}$ on $\Lambda_{i\delta}$ by
$\max_{f_{i\delta} \in \bar \Lambda_{i\delta}} \widehat S_{i}(f_{i\delta}),$
where $\bar \Lambda_{i\delta}$ is a discretization of the FR sphere $\Lambda_{i\delta}$ ($100$ points in our simulation). Next we compute the maximum points $f_{i\delta}^*$ on the discrete FR sphere $\Lambda_{i\delta}$ directly. This means that 
$\widehat S_{i} (f_{i\delta}^*) = \max_{f_{i\delta} \in \bar \Lambda_{i\delta}} \widehat S_{i\delta},$
and since $\widehat S_{i} (f_{i\delta}^*)$ is an estimator built on the sample $\{Y_1,...,Y_N\}$, we can use a bootstrap procedure to obtain a confidence interval for $\max_{f_{i\delta} \in \bar \Lambda_{i\delta}} \widehat S_{i} (f_{i\delta})$. This is expected to give a confidence interval for 
$\max_{f_{i\delta} \in \Lambda_{i\delta}} \widehat S_{i} (f_{i\delta}),$
provided $\bar \Lambda_{i\delta}$ is refined enough.

\section{More details on geometry and proofs}\label{appendix: geometric results}
\subsection{Computation of Fisher-Rao spheres using geodesics} \label{appendix: geodesics and Fisher spheres}

As explained in Section \ref{sec: Fisher-Rao spheres illustration}, the sphere approximation algorithm is based on Proposition \ref{prop: geod spheres}. But this method may fail for large radii. The reasons of this failure may be either the incompleteness of the manifold or finiteness of the injectivity radius, which will be defined in the following.  

\subsubsection{Completeness}
A connected Riemannian manifold $(M,g)$ is geodesically complete if all geodesics $\gamma$ are defined for all time $t\in\mathbb R$. This is equivalent to the completeness of $M$ as a metric space endowed with the geodesic distance $d$, by the Hopf-Rinow theorem \cite{do1992riemannian}. This theorem additionally states that a connected manifold $M$ is geodesically complete if and only if every closed and bounded set in $M$ is compact. This implies that if the manifold is not complete, then there may exist some spheres (which are always closed and bounded) that are not compact. Note that, around a given point, spheres with small enough radii are always compact. Therefore, the non-compactness of spheres may only arise for large enough radii.

For the sphere approximation algorithm from Section \ref{sec: Fisher-Rao spheres illustration}, approximating non-compact spheres can lead to numerical issues. This is because for a given chart on which the geodesic equation is numerically solved, depending on the initial conditions, the solution (the geodesic) $\gamma(t)$ may blow-up in finite time $t_*$. Hence, it will be harder for the approximated solution $\bar \gamma(t)$ to be accurate as $t$ gets closer to $t_*$.
 
\subsubsection{Injectivity radius}
Let $(M,g)$ be a connected Riemannian manifold. If the injectivity radius $\iota_M=\infty,$ then Proposition \ref{prop: geod spheres} can be used for all $\delta>0$. On the contrary, if $\iota_M<\infty$, the radius $\delta$ has to be small enough or the manifold needs to verify some additional properties in order for Proposition \ref{prop: geod spheres} to hold for all $\delta>0$. More precisely, let us first define the exponential map at $p\in M$
\begin{align*}
    \exp_p \ : \ &\Omega_p \subset T_p M \longrightarrow M \\
    & \ v \longmapsto\gamma_{p,v}(1).
\end{align*}
Here, $\gamma_{p,v}$ is the unique geodesic verifying the conditions $\gamma_{p,v}(0)=p$ and $\dot \gamma_{p,v}(0)=v$. The set $\Omega_p$ is defined as the set of vectors $v \in T_p M$ such that $\gamma_{p,v}$ is defined until time $t=1$. It is a neighborhood of the null vector $0_p$ in $T_p M$. In addition, the exponential map $\exp_p$ is a local diffeomorphism in a small neighborhood around $0_p$. 
The question whether $\exp_p$ is a global diffeomorphism is adressed using the notion of injectivity radius. 

Let us first denote $B(p,r)$ and $\mathcal S(p,r)$ resp. the open ball and sphere in $M$ centered at $p$ with radius $r$ for the geodesic distance. Further, let us denote $b(0_p,r):= \{v \in T_p M \ \big| \ |v|_p< r\}$ and $s(0_p,r):= \{v \in T_p M \ \big| \ |v|_p=r\}$ resp. the open ball and sphere centered at $0_p$ with radius $r$ in $T_p M$. Since $\exp_p$ is a local diffeomorphism around $0_p$ for all $p$ in $M$, there exists $r_p>0$ such that for $r\leq r_p$ we have $\exp_p : b(0_p,r_p) \to B(p,r_p)$ is a diffeomorphism. We also have that $\exp_p(s(0_p,r_p)) = \mathcal S(p,r_p)$ and $\exp_p(b(0_p,r_p)) = B(p,r_p)$. The injectivity radius $\iota_p$ at point $p$ is defined as the largest radius $r$ for which $\exp_p$ is a diffeomorphism onto the geodesic ball $B(p,r)\subset M$. That is,
$$\iota_p := \sup\{ r > 0 \ | \ \exp_p \text{ is a diffeomorphism on } b(0_p,r) \}.$$
Further, the injectivity radius of the manifold is defined as $\iota_M:= \inf_{p\in M} \iota_p$. 

If $\iota_M$ is infinite, then $\exp_p : T_p M \to M$ is a global diffeomorphism (for any $p$ in $M$). This implies that for all $r>0$ the spheres $\mathcal S(p,r)$ in $M$ can be identified with
$$\mathcal{S}(p,r)=\exp_p\big(s(0_p,r)\big)=\big\{\gamma(1) \ | \ \gamma \text{ geodesic s.t. } \gamma(0)=p, \dot \gamma(0)=v \text{ and } |v|_p=\delta \big\}.$$
Hence, this justifies the use of the initial algorithm from Section \ref{sec: Fisher-Rao spheres illustration} for all radius $r>0$. This also implies that the manifold is complete.

Now, if the injectivity radius $\iota_M<\infty,$ then there exists $p \in M$ such that for $r\geq \iota_{p}$ the exponential map $\exp_{p}$ fails to be a diffeomorphism from $b(0_{p},r)$ to $B(p,r)$. 
The good news is that if the manifold is complete, the exponential map is always onto $\mathcal S(p,r)$ 
$$\mathcal S(p,r) \subset \exp_p\big( s(0_p,r)\big).$$
This guaranties that the sphere approximation algorithm from Section \ref{sec: Fisher-Rao spheres illustration} gives at least some points that approximate $\mathcal S(p,r)$ (but also some that do not necessarily fall on $\mathcal S(p,r)$). This is true since $M$ is assumed to be connected.
In general, it is not obvious to determine whether $\iota_M$ is infinite or not. Some curvature conditions (Hadamard theorem, see \cite{do1992riemannian}) lead to lower bounds for $\iota_p$. Nevertheless, in practice these conditions are hard to verify. 

\subsection{Proofs} \label{appendix: proof of push-forward}

\begin{proof}[Proof of Lemma \ref{lemma: push forward}]
The fact that $\mathcal R_B$ and $\mathcal P_A$ have the same FIM is a direct consequence of the first assertion and the standard pushforward theorem. Indeed, if we assume $\mathcal R_B=h_\#\mathcal P_A$, then it implies that they have the same FIM. Moreover, since $\mathcal P_A$ and $h_\#\mathcal P_A$ have the same FIM (standard pushforward theorem), the second assertion is obvious.

Now let us prove the first assertion. Let $C$ be a measurable subset on $\mathcal Y$. We have on the one hand
$$h_\#(P_\theta^A)(C) = P_\theta^A(h^{-1}(C)) = \frac{P_\theta(h^{-1}(C)\cap A)}{P_\theta(A)}.$$ 
On the other hand, 
we have
\begin{align*}
    (h_\#P_\theta)^B(C) = \frac{h_\#P_\theta(C\cap B)}{h_\#P_\theta(B)}=\frac{P_\theta(h^{-1}(C\cap B))}{P_\theta(h^{-1}(B))}&=\frac{P_\theta(h^{-1}(C)\cap h^{-1}(B))}{P_\theta(A)}\\
    &=\frac{P_\theta(h^{-1}(C)\cap A)}{P_\theta(A)}.
\end{align*}
This allows to conclude the proof. Note that $h$ does not need be bijective in the proof.
\end{proof}

\begin{proof}[Proof of Lemma \ref{lemma: push forward} in the case of sufficient statistics]
Under the previous notations, if we now assume that $h:\mathcal X \to \mathcal Y$ is a sufficient statistic for $\mathcal P$, then the first assertion of Lemma \ref{lemma: push forward} still holds. Note that in this case, $h$ is also a sufficient statistic of a truncated version of $\mathcal P$ i.e. $\mathcal P_A$ (see \cite{tukey1949sufficiency}). In addition, the Fisher-Neyman factorisation theorem (see \cite{halmos1949application}) implies that
$$p_\theta(x) = u(x) r_\theta\big(h(x)\big),$$
where $p_\theta$ and $r_\theta$ are resp. the pdf of $P_\theta$ and $R_\theta:=h_\# P_\theta$ and $u$ is a non-negative function on $\mathcal X$. Let us denote resp. $p_\theta^A$ and $r_\theta^B$ the truncated pdfs of $p_\theta$ on $A$ and $r_\theta$ on $B$. We can compute the FIM of $K_\theta^B$ of the truncated family $\mathcal R_B=\{R_\theta^B\}_{\theta\in\Theta}$,
\begin{align}
    (K_\theta^B)_{ij} &= \int_\mathcal Y \big(\partial_i \log r_\theta^B(y)\big)\big(\partial_j \log r_\theta^B(y)\big) dR_\theta^B(y) \nonumber\\
     &= \int_\mathcal X  \big(\partial_i \log r_\theta^B(h(x))\big)\big(\partial_j \log r_\theta^B(h(x))\big) dP_\theta^A(x) \\
     &= \int_\mathcal X  \left(\partial_i \log \frac{p_\theta(x)}{u(h(x))R_\theta(B)}\right)\left(\partial_j \log \frac{p_\theta(x)}{u(h(x))R_\theta(B)}\right) dP_\theta^A(x)  \\
     &= \int_\mathcal X  \left(\partial_i \log \frac{p_\theta(x)}{P_\theta(A)}\right)\left(\partial_j \log \frac{p_\theta(x)}{P_\theta(A)}\right) dP_\theta^A(x) \\
     &= \int_\mathcal X  \left(\partial_i \log p_\theta^A(x)\right)\left(\partial_j \log p_\theta^A(x)\right) dP_\theta^A(x) \\
     &= (I_\theta^A)_{ij}, \nonumber
\end{align}
where (B1) is obtained from the transfer theorem, (B2) and (B4) are resp. obtained by the definition of $r_\theta^B$ and $p_\theta^A$ as well as the Fisher-Neyman theorem and (B3) holds because $u\big(h(x)\big)$ does not depend on $\theta$ and $R_\theta(B)=P_\theta(A)$.    
\end{proof}

\begin{proof}[Proof of Proposition \ref{prop: loc scale FIM}]
    Let $p_\theta:\mathbb R \to (0,\infty)$ be defined by 
$$ p_\theta(x) = \frac{1}{s} p\left(\frac{x-m}{s}\right),$$
where $\theta=(m,s) \in \mathbb R\times (0,\infty)$ is the parameter. Here, $p$ is the initial density. Let us compute the FIM for this family. First, taking the partial derivatives of $\log p_\theta$ we get
$$\partial_m \log p_\theta(x) = \frac{-\frac{1}{s} p'\left( \frac{x-m}{s}\right)}{p\left( \frac{x-m}{s}\right)} \ \ \ \text{   and   } \ \ \ \partial_s \log p_\theta(x) =\left( -\frac{x-m}{s^2}\right) \frac{p'\left( \frac{x-m}{s}\right)}{p\left( \frac{x-m}{s}\right)}  -\frac{1}{s}.$$

Now, we can compute 
\begin{align*}
    (I_\theta)_{11} = \int_{\mathbb R} \left(\frac{-\frac{1}{s} p'\left( \frac{x-m}{s}\right)}{p\left( \frac{x-m}{s}\right)}\right)^2 \frac{1}{s}p\left( \frac{x-m}{s}\right) dx
    = \frac{1}{s^2}\alpha
\end{align*} 
\begin{align*}
    (I_\theta)_{12}
    &= \int_{\mathbb R} \left(\frac{-\frac{1}{s} p'\left( \frac{x-m}{s}\right)}{p\left( \frac{x-m}{s}\right)}\right) \left( \left( -\frac{x-m}{s^2}\right) \frac{p'\left( \frac{x-m}{s}\right)}{p\left( \frac{x-m}{s}\right)}  -\frac{1}{s}\right) \frac{1}{s}p\left( \frac{x-m}{s}\right) dx= \frac{1}{s^2}\gamma
\end{align*}
\begin{align*}
    (I_\theta)_{22} = \int_{\mathbb R} \left(\left( -\frac{y}{s}\right) \frac{p'(y)}{p(y)}  -\frac{1}{s} \right)^2 \frac{1}{s}p(y) s\ dy
    =\frac{1}{s^2}\beta.
\end{align*}
Here $\alpha,\beta,\gamma$ are given by
$$\alpha=\int_{\mathbb R} \left(\frac{ p'(y)}{p(y)}\right)^2 p(y)\ dy, \ \ \ \beta=\int_{\mathbb R} \left(y\frac{p'(y)}{p(y)}  + 1\right)^2 p(y) dy$$
and 
$$\gamma = \int_{\mathbb R} \left(\frac{-p'(y)}{p(y)}\right) \left(  -y\frac{p'(y)}{p(y)} - 1 \right) p(y) dy.$$
We refer to \cite{komaki2007bayesian} for a similar computation in higher dimension. We assume that these quantities are well defined and finite. In this case, the FIM for the location-scale family of $p$ is given by 
$$ I_\theta = \frac{1}{s^2}\begin{bmatrix} \alpha & \gamma \\ \gamma & \beta \end{bmatrix}.$$
The matrix $\begin{bmatrix} \alpha & \gamma \\ \gamma & \beta \end{bmatrix}$ can be diagonalized on an orthonormal basis of eigenvectors
$$I_\theta =  Q_{\alpha\beta\gamma}\left(\frac{1}{s^2}\begin{bmatrix}
    \lambda_1 & 0 \\
    0 & \lambda_2
\end{bmatrix}\right) Q_{\alpha\beta\gamma}^\intercal = Q_{\alpha\beta\gamma}\begin{bmatrix}
    \sqrt{\lambda_1} & 0 \\
    0 & \sqrt{\lambda_2}\end{bmatrix}
    \left(\frac{1}{s^2}\begin{bmatrix}
    1 & 0 \\
    0 & 1\end{bmatrix}\right)
    \begin{bmatrix}
    \sqrt{\lambda_1} & 0 \\
    0 & \sqrt{\lambda_2}\end{bmatrix}Q_{\alpha\beta\gamma}^\intercal.$$ 
Here, $Q_{\alpha\beta\gamma}$ is an orthogonal change-of-basis matrix and $\lambda_1,\lambda_2\geq0$ are the eigenvalues and they depend on $\alpha,\beta,\gamma$. By denoting $P_{\alpha\beta\gamma} := \begin{bmatrix}
    \sqrt{\lambda_1} & 0 \\
    0 & \sqrt{\lambda_2}\end{bmatrix}Q_{\alpha\beta\gamma}$ we thus get 
$$I_\theta = P_{\alpha\beta\gamma}  \left(\frac{1}{s^2}\begin{bmatrix}
    1 & 0 \\
    0 & 1\end{bmatrix}\right) P_{\alpha\beta\gamma}^\intercal.$$
This proves that the FIM is, up to reparameterization, the Poincaré metric on $\mathbb H := \{(x,y) \in \mathbb R^2 \ | \ y >0\}$.
\end{proof}

For any location-scale family, the change-of-basis matrix $P_{\alpha\beta\gamma}$ depends on $\alpha,\beta,\gamma$ therefore having a good approximation of these integral quantities will allow us to approximate $P_{\alpha \beta \gamma}$. This will in turn allow us to compute geometric quantities (distance, curvature, angle,...) for the location-scale family using only the Poincaré metric. Indeed, these quantities are already known for the latter.
    
\subsection{Explicit Fisher-Rao distance computation in the triangular and exponential distributions} \label{appendix: proof of exponential and triangular}
The triangular and exponential distribution families only depend on a single parameter. We can thus easily compute the Fisher information and the Fisher-Rao distance exlicitly.
\subsubsection{Triangular family}
The triangular family $\mathcal T_{[a,b]}$ of densities $q_m$ on the interval $[a,b]$ is given by
$$q_m(x)= \frac{2(x-a)}{(b-a)(m-a)}\textbf{1}_{x\in [a,m]} + \frac{2(b-x)}{(b-a)(b-m)}\textbf{1}_{x\in (m,b]}, $$
where $m\in(a,b)$ is the parameter. The Fisher information $i_m$ is computed as follows: \begin{itemize}
    \item for $x\in [a,m]$, 
    $\partial_m \log q_m(x) =  - \partial_m \log(m-a) = -\frac{1}{m-a}, $
    \item for $x \in (m,b]$,
    $\partial_m \log q_m(x) = -\partial_m \log(b-m) = \frac{1}{b-m}.$
\end{itemize}
Hence, we have
\begin{align*}
    i_m &= \int_a^b \big(\partial_m \log q_m(x)\big)^2 q_m(x) dx \\
        &= \int_a^m \frac{1}{(m-a)^2} \frac{2(x-a)}{(b-a)(m-a)} dx + \int_m^b \frac{1}{(b-m)^2} \frac{2(b-x)}{(b-a)(b-m)} dx \\
        &= \frac{1}{(m-a)^3(b-a)}(m-a)^2 + \frac{1}{(b-m)^3(b-a)}(b-m)^2\\
        &= \frac{1}{b-a}\left(\frac{1}{m-a} + \frac{1}{b-m} \right) = \frac{1}{(m-a)(b-m)}.
\end{align*}
As $\{q_m : m \in (a,b)\}$ is a one parameter family, every geodesic correspond to a segment $\gamma(t) = m_0 + tv$, up to reparameterization. But, since the integral in the definition of the Fisher-Rao distance 
is invariant under reparameterization, we can explicitly compute the distance between two points using the segment $\gamma(t)~=~m_0~+~tv$. To begin with, let $\alpha:= \frac{(a+b)}{2}$ and $\beta := \left(\frac{a-b}{2}\right)^2$. We have $(b-m)(m-a) = -\left(m-\alpha \right)^2 + \beta.$ Then, for $m_0,m_1 \in (a,b)$, we compute 
\begin{align*}
d(m_0,m_1) &=\left| \int_{m_0}^{m_1} \frac{dm}{\sqrt{-(m-\alpha)^2 + \beta}}\right| = \left|\int_{m_0}^{m_1}\frac{dm}{\sqrt{\beta}\sqrt{-\left(\frac{m-\alpha}{\sqrt{\beta}}\right)^2 + 1}} \right|  \\
&\overset{u:=\frac{m-\alpha}{\sqrt{\beta}}}{=} \left|\int_{\frac{m_0-\alpha}{\sqrt{\beta}}}^{\frac{m_1-\alpha}{\sqrt{\beta}}} \frac{du}{\sqrt{1-u^2}}\right| \\
&=\left| \arcsin\left(\frac{m_1-\alpha}{\sqrt{\beta}}\right)-\arcsin\left(\frac{m_0-\alpha}{\sqrt{\beta}}\right)\right|.    
\end{align*}
This means that the Fisher-Rao sphere centered at $m_0 \in (a,b)$ with small radius $\delta>0$, is explicitly determined by the following two points
$$
\begin{cases}
m_+ = \sqrt{\beta} \sin\left(\delta + \arcsin\left(\frac{m_0-\alpha}{\sqrt{\beta}}\right) \right) + \alpha, \ \ \ m\geq m_0, \\
m_- = \sqrt{\beta} \sin\left(\arcsin\left(\frac{m_0-\alpha}{\sqrt{\beta}}\right) - \delta \right) + \alpha, \ \ \ m < m_0.
\end{cases}
$$

\subsubsection{Exponential distributions family}

The exponential family $\mathcal E$ has densities given by
$$ p_\lambda(x) = \lambda e^{-\lambda x} \textbf{1}_{x>0}, \ \ \  \lambda \in (0,\infty).$$
The Fisher information is given by $i_\lambda = \lambda^{-2}$ (see \cite{calin2014geometric}, page 25). Similar to the triangular family, the Fisher-Rao geodesic distance in this case is
$$ d(\lambda_0,\lambda_1) = \left| \int_{\lambda_0}^{\lambda_1} \sqrt{\frac{1}{\lambda^2}} d\lambda \right| =  \left|\log\left(\frac{\lambda_1}{\lambda_0}\right) \right|, \ \ \ \lambda_0,\lambda_1>0.$$
So that, the Fisher-Rao sphere centered at $\lambda_0$ with radius $\delta$ is $\{\lambda_-,\lambda_+\}$ given by
$$\begin{cases}
\lambda_+ = \lambda_0 e^{\delta}, \ \ \ \lambda_+\geq \lambda_0,\\
\lambda_- = \lambda_0 e^{-\delta}, \ \ \ \lambda_-\leq \lambda_0.\\
\end{cases}$$
Note that this family has infinite diameter since 
$$\text{diam}(\mathcal E) := \sup_{\lambda_0,\lambda_1 \in (0,\infty)} d(\lambda_0,\lambda_1)=\sup_{\lambda_0,\lambda_1 \in (0,\infty)} \left|\log\left(\frac{\lambda_1}{\lambda_0}\right) \right|=\infty. $$

\pagebreak

\section{Dualistic structure in information geometry} \label{appendix: information geometry}
As explained in the main article, the Riemannian structure of $\mathcal P$ allows to obtain the Levi-Civita affine connection $\bar \nabla$. In addition, Eguchi's formulas \cite{eguchi1985differential} endows the space $\mathcal P$ with an additional dualistic structure:
\begin{align*}
    \Gamma_{ij,k} (\theta_0) =  - \partial_{\theta^i}\partial_{\theta^j}\partial_{\theta_0^k} \Big[ \mathcal D(P_\theta | P_{\theta_0}) \Big]_{\theta=\theta_0},\\
    ^*\Gamma_{ij,k} (\theta_0) =  - \partial_{\theta^k}\partial_{\theta_0^i}\partial_{\theta_0^j} \Big[ \mathcal D(P_\theta | P_{\theta_0}) \Big]_{\theta=\theta_0}.
\end{align*}
The coefficients $\Gamma_{ij,k}$ and $^*\Gamma_{ij,k}$ are the so-called Christoffel symbols (of the first kind) of two affine connections on $\mathcal P$ denoted by $\nabla$ and $\nabla^*$ that verify a dualistic (conjugacy) property \cite{nielsen2020elementary}. This property implies that the unique Levi-Civita connection $\bar \nabla$ (resp. Christoffel symbols $\bar \Gamma$) of the Fisher metric $g$ is given by the average of the two dual connections (resp. Christoffel symbols)
$$\bar \nabla = \frac{\nabla \ + \ ^*\nabla}{2} \ \ \ \text{and} \ \ \ \bar \Gamma_{ij,k} = \frac{\Gamma_{ij,k} \ + \ ^*\Gamma_{ij,k}}{2}.$$
The Chistoffel symbols of the first kind $\bar \Gamma_{ij,k}$ and the Christoffel symbols (of the second kind) $\bar \Gamma_{ij}^k$ defined in the main article are related by the following formula
$$ \bar \Gamma_{ij}^k = \sum_{l=1}^m \bar \Gamma_{ij,l}\ g^{lk}.$$
Here, $g^{lk}$ is the inverse matrix coefficients of $g$. These two previous formulas provide another way for computing the Levi-Civita Christoffel symbols. This will be useful in Section \ref{appendix: computing FIM and Christoffels} for the truncated normal family.

\section{Computations of FIM and Christoffel symbol coefficients} \label{appendix: computing FIM and Christoffels}

We will now compute the FIM coefficients and the Christoffel symbols for some families of probability distributions. This will prove Proposition 3.1 from the main article. These computations are motivated by the fact that we need to compute the Christoffel symbols in order to approximate the Fisher-Rao spheres.

\subsection{Truncated normal family} 
We will explicitly compute the FIM $I_\theta^{[a,b]}$ and Christoffel symbols $\bar \Gamma_{ij,k}$ of the truncated normal family $\mathcal N_{[a,b]}$. We use the formula \cite{nielsen2020elementary}
\begin{align}\label{eq: FIM with KL div}
    (I_\theta^{[a,b]})_{ij} = -\partial_{\theta^i}\partial_{\theta_0^j} \Big[\mathcal D(q_\theta | q_{\theta_0})\Big]_{\theta=\theta_0}.
\end{align}
Here, $I_\theta^{[a,b]}$ is the FIM and $q_\theta$ is the truncated density and $\mathcal D$ is the Kullback-Leibler divergence defined as $\mathcal D(q_\theta| q_{\theta_0}):= \mathbb E_{X\sim P_\theta} \left[\log\left({q_\theta}/{q_{\theta_0}}(X)\right)\right]$. For the Christoffel symbols $\bar \Gamma_{ij,k}$, we first compute $\Gamma_{ij,k}$ and $ \Gamma_{ij,k}^*$ (see Section \ref{appendix: information geometry} for explanation)
$$ \Gamma_{ij,k} (\theta_0) =  - \partial_{\theta^i} \partial_{\theta^j} \partial_{\theta_0^k} \big[ \mathcal D(q_\theta | q_{\theta_0})\big]_{\theta=\theta_0} \ \ \text{ and } \ \ ^* \Gamma_{ij,k}  (\theta_0) =  - \partial_{\theta^k} \partial_{\theta_0^i} \partial_{\theta_0^j} \big[ \mathcal D(q_\theta | q_{\theta_0}) \big]_{\theta=\theta_0}.$$
Then, we take the average to get the Levi-Civita Christoffel symbols $\bar \Gamma_{ij,k}.$

\subsubsection{Computing the conditional expectation and variance of Gaussian variables}
Let $B=[a,b]$, it denotes the truncation domain. Denoting $P_\theta =\mathcal N(\mu,\sigma^2)$, for $\theta=(\mu,\sigma)$, with density $p_\theta$, the conditional expectation and variance over $B$ of a Gaussian random variable is given by 
$$ \mu_B := \mathbb E_{X\sim q_\theta}(X) =  \frac{1}{P_\theta(B)}\int_B x p_\theta(x) dx = \frac{-\sigma^2}{P_\theta(B)} \int_B \frac{-(x -\mu)}{\sigma^2} p_\theta(x) dx + \mu = -\sigma^2 \big[q_\theta\big]_a ^b + \mu, $$
 and
\begin{align*}
    \sigma_B^2  &:= \text{Var}_{X\sim q_\theta}(X) \\
    &=\frac{1}{P_\theta(B)} \int_B (x-\mu_B)^2p_\theta(x) dx \\
    &= \frac{1}{P_\theta(B)} \left( \int_B (x-\mu)^2p_\theta(x) dx  + 2(\mu - \mu_B)\int_B (x-\mu) p_\theta(x) dx + \int_B (\mu-\mu_B)^2p_\theta(x) dx  \right) \\
    &= \frac{-\sigma^2}{P_\theta(B)} \left(\big[ (x-\mu)p_\theta\big]_a^b -\int_B p_\theta(x)dx   \right) -2(\mu - \mu_B)^2 + (\mu-\mu_B)^2\\
    &= -\sigma^2 \left( \big[ (x-\mu)q_\theta\big]_a^b -1 \right) - (\mu - \mu_B)^2.
\end{align*}

\subsubsection{Computing the Kullback-Leibler divergence of truncated Gaussian densities}
Let us compute the Kullback-Leibler divergence on the truncated normal family 
$$ D(\theta,\theta_0) := \mathcal D( q_\theta | q_{\theta_0}) = \int q_\theta \log \left(\frac{q_\theta}{q_{\theta_0}} \right) = \int q_\theta \log q_\theta - \int q_\theta \log q_{\theta_0}. $$
Setting $ A_\theta := \int q_\theta \log q_\theta,$ we have
\begin{align}
    D(\theta,\theta_0) &= A_\theta - \int q_\theta \log q_{\theta_0} \nonumber \\
    &= A_\theta - \int_B \frac{p_\theta}{P_\theta(B)} \log p_{\theta_0} + \int_B \frac{p_\theta}{P_\theta(B)} \log P_{\theta_0}(B) \nonumber \\
    & = A_\theta - \int_B \frac{p_\theta}{P_\theta(B)} \log p_{\theta_0} + \underbrace{\log P_{\theta_0}(B)}_\textrm{$=:C_{\theta_0}$} \nonumber \\
    & = A_\theta + C_{\theta_0} - \frac{1}{P_\theta(B)}\int_B p_\theta \log p_{\theta_0} \label{divergence} ,
\end{align}
$A_\theta$ and $C_{\theta_0}$ depend on $\theta$ and $\theta_0$ respectively. Let us now compute the last term explicitly:
\begin{align*}
    \int_B p_\theta \log p_{\theta_0} &= \int_B p_\theta \left( -\log (\sqrt{2\pi} \sigma_0) - \frac{(x-\mu_0)^2}{2\sigma_0^2} \right) \\
    & = -\log (\sqrt{2\pi} \sigma_0) P_\theta(B) -\frac{1}{2\sigma_0^2} \int_B (x-\mu_0)^2 p_\theta. 
\end{align*}
We finally obtain
\begin{align*}
    \frac{1}{2\sigma_0^2}\int_B (x-\mu_0)^2 p_\theta &= \frac{1}{2\sigma_0^2} \left(\int_B (x-\mu_B)^2 p_\theta + 2(\mu_B - \mu_0) \underbrace{\int_B (x-\mu_B)p_\theta}_{=0} +(\mu_B - \mu_0)^2 P_\theta(B) \right)\\
    & = P_\theta(B) \frac{\sigma_B^2}{2\sigma_0^2} +\frac{(\mu_B - \mu_0)^2}{2\sigma_0^2} P_\theta(B).
\end{align*}

Now we can express (\ref{divergence}) using $\mu_B$ and $\sigma_B^2$ (which depend on $\theta$)
$$D(\theta,\theta_0) = A_\theta + C_{\theta_0}  + \log(\sqrt{2\pi}) + \log \sigma_0 + \frac{\sigma_B^2 + (\mu_B - \mu_0)^2}{2\sigma_0^2}.$$

\subsubsection{Computing the Fisher Information matrix} \label{sec: FIM of trunc normal family}

In the computation of the FIM using (\ref{eq: FIM with KL div}), we only end up with terms that depend both on $\theta_0$ and $\theta$. This means that $A_\theta$ and $C_{\theta_0}$ disappear. Therefore, we omit them from the start. For $\partial_{\mu_0} D$ and $\partial_{\sigma_0} D$ we obtain 
$$ \partial_{\mu_0} D = \frac{2(-1)(\mu_B - \mu_0)}{2\sigma_0^2}= -\frac{\mu_B - \mu_0}{\sigma_0^2} \ \ \ \text{ and } \ \ \ \partial_{\sigma_0} D = \frac{1}{\sigma_0} - \frac{\sigma_B^2 + (\mu_B -\mu_0)^2}{\sigma_0^3}.$$
Now for $\partial_\mu \partial_{\mu_0} D $, $\partial_\sigma \partial_{\mu_0} D $ and $\partial_\sigma \partial_{\sigma_0} D $, we get
\begin{align} 
    \partial_\mu \partial_{\mu_0} D  &= - \frac{ \partial_\mu \mu_B(\theta) }{\sigma_0^2},\label{metrque}\\
    \partial_\sigma \partial_{\mu_0} D &= -\frac{\partial_\sigma \mu_B(\theta)}{\sigma_0^2},\label{metrque2}\\
\partial_\sigma \partial_{\sigma_0} D &= \frac{-1}{\sigma_0^3} \big( \partial_\sigma \sigma_B^2(\theta) + 2(\mu_B - \mu_0) \partial_\sigma \mu_B(\theta) \big).\label{metrque3}    
\end{align} 

Let us now compute the terms $\partial_\mu \mu_B(\theta)$, $\partial_\sigma \mu_B (\theta)$ and $\partial_\sigma \sigma_B^2 (\theta)$:
\begin{align*} \partial_\mu \mu_B(\theta) &= 1 - \sigma^2 \big[\partial_\mu q_\theta(x) \big]_a^b = 1 - \sigma^2 \left[ q_\theta(x)\left( \frac{x-\mu}{\sigma^2} + \big[q_\theta\big]_a^b\right) \right]_a^b,\\
\partial_\sigma \mu_B(\theta) &= -2\sigma \big[ q_\theta(x) \big]_a^b -\sigma^2 \big[ \partial_\sigma q_\theta(x) \big]_a^b\\
&= - 2 \sigma\big[q_\theta\big]_a^b  -\sigma\left[ q_\theta(x)\left([(x- \mu)q_\theta]_a^b - 1  +\frac{(x-\mu)^2}{\sigma^2}\right) \right]_a^b\\
&= -\sigma \left[ q_\theta(x)\left([(x- \mu)q_\theta]_a^b + 1  +\frac{(x-\mu)^2}{\sigma^2}\right) \right]_a^b
\partial_\sigma \\
\sigma_B^2(\theta) &= 2\sigma - 2\sigma\big[ (x-\mu)q_\theta \big]_a^b - \sigma^2\big[ (x-\mu)\partial_\sigma q_\theta \big]_a^b - 2(\mu-\mu_B)\partial_\sigma \mu_B(\theta) \\
    &= 2\sigma \left( 1 - \big[ (x-\mu)q_\theta \big]_a^b \right) - \sigma^2\left[ (x-\mu)q_\theta(x)\left( \frac{x-\mu}{\sigma^2} + \big[q_\theta\big]_a^b\right) \right]_a^b\\ & - 2(\mu-\mu_B)\partial_\sigma \mu_B(\theta). \\
\end{align*}

Lastly, coming back to (\ref{metrque}), (\ref{metrque2}) and (\ref{metrque3}) we get
\begin{align*} 
    \partial_\mu \partial_{\mu_0} D  &= - \frac{ \partial_\mu \mu_B(\theta) }{\sigma_0^2} = - \frac{1}{\sigma_0^2} + \frac{\sigma^2 \left[ q_\theta(x)\left( \frac{x-\mu}{\sigma^2} + \big[q_\theta\big]_a^b\right) \right]_a^b}{\sigma_0^2},\\
    \partial_\sigma \partial_{\mu_0} D &= -\frac{\partial_\sigma \mu_B(\theta)}{\sigma_0^2} = \frac{\sigma \left[ q_\theta(x)\left([(x- \mu)q_\theta]_a^b + 1  +\frac{(x-\mu)^2}{\sigma^2}\right) \right]_a^b}{\sigma_0^2},\\ 
\partial_\sigma \partial_{\sigma_0} D &= \frac{-1}{\sigma_0^3} \big( \partial_\sigma \sigma_B^2(\theta) + 2(\mu_B - \mu_0) \partial_\sigma \mu_B(\theta) \big) \\
 &= \frac{-1}{\sigma_0^3} \left(2\sigma \left( 1 - \big[ (x-\mu)q_\theta \big]_a^b \right) - \sigma^2\left[ (x-\mu)q_\theta(x)\left( \frac{x-\mu}{\sigma^2} + \big[q_\theta\big]_a^b\right) \right]_a^b  \right).     
\end{align*} 
Further, taking $\theta=\theta_0$ and multiplying by $-1$, we obtain
\begin{align*} 
    (I_{\theta_0})_{11}=-\partial_\mu \partial_{\mu_0} D_{|\theta=\theta_0} &=  \frac{1}{\sigma_0^2} - \left[ q_{\theta_0}(x)\left( \frac{x-\mu_0}{\sigma_0^2} + \big[q_{\theta_0}\big]_a^b\right) \right]_a^b,\\
    (I_{\theta_0})_{12} =-\partial_\sigma \partial_{\mu_0} D_{|\theta=\theta_0} &= \frac{-\sigma_0 \left[ q_{\theta_0}(x)\left([(x- \mu_0)q_{\theta_0}]_a^b + 1  +\frac{(x-\mu_0)^2}{\sigma_0^2}\right) \right]_a^b}{\sigma_0^2},\\
(I_{\theta_0})_{22}=-\partial_\sigma \partial_{\sigma_0} D_{|\theta=\theta_0}
 &= \frac{2}{\sigma_0^2} - \frac{2 \big[ (x-\mu_0)q_{\theta_0} \big]_a^b }{\sigma_0^2} -  \left[ \frac{(x-\mu_0)}{\sigma_0} q_{\theta_0}(x)\left( \frac{x-\mu_0}{\sigma_0^2} + \big[q_{\theta_0}\big]_a^b\right) \right]_a^b  .     
\end{align*} 
\subsubsection{A few remarks on the truncated normal family}\label{appendix: rmk on trunc norm fam}
Here, we observe that when $a\to -\infty$ and $b\to \infty$, we get
$$(I_{\theta_0})_{11} \to \frac{1}{\sigma_0^2}, \ \ \  (I_{\theta_0})_{12} \to 0 \ \ \ \text{and} \ \ \   (I_{\theta_0})_{22} \to \frac{2}{\sigma_0^2}.$$
In other words, $I_\theta^{[a,b]}$ converges to the FIM of the non-truncated normal family $I_\theta$. 

Now, if we fix $a$ and $b$, for $\theta=(\mu,\sigma)$ such that $\mu \in (a,b)$ and $\sigma$ is close to $0$, then the normalization constant $N_{\theta}:= \int_a^b (\sqrt{2\pi}\sigma)^{-1}\exp(-(x-\mu)^2/(2\sigma^{2}) dx$
is close to $1$. Indeed, most of the mass of the normal density $p_\theta$ is already inside $[a,b]$. As a consequence, $q_\theta$ and $\mathcal N(\mu,\sigma^2)$ are close. So that, for such $\theta$, the matrices $I_\theta$ and $I_\theta^{[a,b]}$ are also close. This means that in areas where $\mu \in (a,b)$ and $\sigma$ is close to $0$, both of these metrics induce a very similar geometry.

\subsubsection{Computing the Levi-Civita Christoffel symbols}
As explained previously, the Levi-Civita Christoffel symbols $\bar \Gamma_{ij,k}$ is computed by averaging the Christoffel symbols of $\Gamma_{ij,k}$ and $^*\Gamma_{ij,k}$
$$ \bar \Gamma_{ij,k} = \frac{\Gamma_{ij,k} +  ^*\Gamma_{ij,k}}{2}.$$
It remains to express it as a function of $\mu_B$ and $\sigma_B^2$ and of their partial derivatives w.r.t. $\mu$ and $\sigma$.
\paragraph{Computing $^* \Gamma_{ij,k}$ }

First, for $\partial_{\mu_0} D$ and $\partial_{\sigma_0} D$ we have
\begin{align*}
    \partial_{\mu_0} D &= -\frac{\mu_B - \mu_0}{\sigma_0^2},\\
    \partial_{\sigma_0} D &= \frac{1}{\sigma_0} - \frac{\sigma_B^2 + (\mu_B -\mu_0)^2}{\sigma_0^3}.
\end{align*} 

We differentiate these two expressions with respect to $\partial_{\mu_0}$ and $\partial_{\sigma_0}$ to obtain $\partial_{\mu_0} \partial_{\mu_0} D$, $\partial_{\mu_0} \partial_{\sigma_0} D$ and $\partial_{\sigma_0} \partial_{\sigma_0} D$ :
\begin{align*}
\partial_{\mu_0} \partial_{\mu_0} D &= \frac{1}{\sigma_0^2},\\
\partial_{\mu_0} \partial_{\sigma_0} D &= \frac{2(\mu_B - \mu_0)}{\sigma_0^3},\\
\partial_{\sigma_0} \partial_{\sigma_0} D &= \frac{-1}{\sigma_0^2} + 3\frac{\sigma_B^2 + (\mu_B - \mu_0)^2}{\sigma_0^4}.
\end{align*}

Now, it remains to differentiate these three expressions w.r.t. to $\mu$ and $\sigma$ to obtain the following six functions $ \partial_\mu \partial_{\mu_0} \partial_{\mu_0} D$, $\partial_\sigma \partial_{\mu_0} \partial_{\mu_0} D$,  $\partial_\mu \partial_{\mu_0} \partial_{\sigma_0} D$, $\partial_\sigma \partial_{\mu_0} \partial_{\sigma_0} D$, $\partial_\mu \partial_{\sigma_0} \partial_{\sigma_0} D$ and $\partial_\sigma \partial_{\sigma_0} \partial_{\sigma_0} D$. We obtain 
\begin{align*}
  \partial_\mu \partial_{\mu_0} \partial_{\mu_0} D &= 0,\\  
  \partial_\sigma \partial_{\mu_0} \partial_{\mu_0} D &=  0,\\
  \partial_\sigma \partial_{\mu_0} \partial_{\sigma_0} D &=  \frac{2 \partial_\sigma [\mu_B(\theta) ]}{\sigma_0^3},\\
  \partial_\mu \partial_{\sigma_0} \partial_{\sigma_0} D &= \frac{3}{\sigma_0^4} \Big( \partial_\mu [\sigma_B^2(\theta) ] + 2 \partial_\mu [\mu_B(\theta)] (\mu_B-\mu_0) \Big),\\
  \partial_\sigma \partial_{\sigma_0} \partial_{\sigma_0} D &= \frac{3}{\sigma_0^4} \Big( \partial_\sigma [\sigma_B^2(\theta) ] + 2 \partial_\sigma [\mu_B(\theta)] (\mu_B-\mu_0) \Big).
\end{align*}

Let us now compute the partial derivatives of $\mu_B$ and $\sigma_B^2$ in terms of the truncated Gaussian density $q_\theta$
 \begin{align*}
\partial_\mu \mu_B(\theta) &= 1 - \sigma^2 \big[ \partial_\mu q_\theta \big]_a^b,\\
\partial_\sigma \mu_B(\theta) &= -2\sigma \big[ q_\theta \big]_a^b - \sigma^2 \big[ \partial_\sigma q_\theta \big]_a^b,\\
    \partial_\mu \sigma_B^2 (\theta) &= - \sigma^2  \big[- q_\theta +(x-\mu) \partial_\mu q_\theta \big]_a^b - 2 (\mu - \mu_B) \Big( 1 - \partial_\mu \big[\mu_B(\theta)\big] \Big),\\
    \partial_\sigma \sigma_B^2(\theta) &= 2\sigma \left( 1 - \big[ (x-\mu)q_\theta \big]_a^b \right) - \sigma^2\Big[ (x-\mu)\partial_\sigma q_\theta \Big]_a^b  + 2(\mu-\mu_B)\partial_\sigma \mu_B(\theta). 
\end{align*}

Coming back to the symbols $^* \Gamma_{ij,k}$, we have
\begin{align*}  
^* \Gamma_{11,1}&=  -\partial_\mu \partial_{\mu_0} \partial_{\mu_0} D_{|\theta = \theta_0} = 0, \\
^* \Gamma_{11,2}&= -\partial_\sigma \partial_{\mu_0} \partial_{\mu_0} D_{|\theta = \theta_0} =  0,\\
^* \Gamma_{12,1} &= \ ^*\Gamma_{21,1}= -\partial_\mu \partial_{\mu_0} \partial_{\sigma_0} D_{|\theta = \theta_0} = - \frac{2 \partial_\mu [\mu_B(\theta_0) ]}{\sigma_0^3},\\
^* \Gamma_{12,2}&=\ ^* \Gamma_{21,2}= -\partial_\sigma \partial_{\mu_0} \partial_{\sigma_0} D_{|\theta = \theta_0} =  -\frac{2 \partial_\sigma [\mu_B(\theta_0) ]}{\sigma_0^3},\\
^* \Gamma_{22,1} &= -\partial_\mu \partial_{\sigma_0} \partial_{\sigma_0} D_{|\theta = \theta_0} = -\frac{3}{\sigma_0^4} \Big( \partial_\mu [\sigma_B^2(\theta_0) ] + 2 \partial_\mu [\mu_B(\theta_0)] (\mu_B(\theta_0)-\mu_0) \Big),\\
^* \Gamma_{22,2}&= -\partial_\sigma \partial_{\sigma_0} \partial_{\sigma_0} D_{|\theta = \theta_0} = -\frac{3}{\sigma_0^4} \Big( \partial_\sigma [\sigma_B^2(\theta_0) ] + 2 \partial_\sigma [\mu_B(\theta_0)] (\mu_B(\theta_0)-\mu_0) \Big).
\end{align*}

\paragraph{Computing $\Gamma_{ij,k}$}

We have already computed  $\partial_{\mu_0} D$ and $\partial_{\sigma_0} D$ 
\begin{align*}
\partial_{\mu_0} D = -\frac{\mu_B - \mu_0}{\sigma_0^2} \ \ \ \text{and} \ \ \ 
\partial_{\sigma_0} D = \frac{1}{\sigma_0} - \frac{\sigma_B^2 + (\mu_B -\mu_0)^2}{\sigma_0^3}.
\end{align*}
We also computed $\partial_\mu \partial_{\mu_0} D$, $\partial_\sigma \partial_{\mu_0} D$ and $\partial_\sigma \partial_{\sigma_0} D$
\begin{align*}
    \partial_\mu \partial_{\mu_0} D  &= - \frac{ \partial_\mu \mu_B(\theta) }{\sigma_0^2},\\
    \partial_\sigma \partial_{\mu_0} D &= -\frac{\partial_\sigma \mu_B(\theta)}{\sigma_0^2},\\
    \partial_\sigma \partial_{\sigma_0} D &= \frac{-1}{\sigma_0^3} \big( \partial_\sigma \sigma_B^2(\theta) + 2(\mu_B - \mu_0) \partial_\sigma \mu_B(\theta) \big).    
\end{align*} 
It remains to compute $ \partial_\mu\partial_{\sigma_0} D$
$$\partial_\mu\partial_{\sigma_0}  D = \frac{-1}{\sigma_0^3}\Big( \partial_\mu [\sigma_B^2]  + 2 \partial_\mu[\mu_B](\mu_B - \mu_0) \Big). $$
Now differentiating the last four expressions with respect to $\partial_\mu $ and $\partial_\sigma$, we obtain
\begin{align*}
    \partial_\mu \partial_\mu \partial_{\mu_0} D  &= - \frac{ \partial_\mu^2 \mu_B(\theta) }{\sigma_0^2},\\
    \partial_\sigma \partial_\mu \partial_{\mu_0} D  &= - \frac{ \partial_{\sigma\mu} ^2 \mu_B(\theta) }{\sigma_0^2},\\
    \partial_\sigma \partial_\sigma \partial_{\mu_0} D &= -\frac{\partial_\sigma^2 \mu_B(\theta)}{\sigma_0^2},\\
    \partial_\mu \partial_\mu \partial_{\sigma_0} D &= \frac{-1}{\sigma_0^3} \Big( \partial_\mu^2 \sigma_B^2(\theta) + 2(\mu_B - \mu_0) \partial_\mu^2 \mu_B(\theta) + 2\big(\partial_\mu \mu_B(\theta) \big)^2 \Big),\\
\partial_\mu \partial_\sigma \partial_{\sigma_0} D &= \frac{-1}{\sigma_0^3} \Big( \partial_\mu \partial_\sigma \sigma_B^2(\theta) + 2(\mu_B - \mu_0) \partial_{\mu\sigma}^2 \mu_B(\theta) + 2\big(\partial_\sigma \mu_B(\theta) \big)\big(\partial_\mu \mu_B(\theta) \big) \Big),\\
\partial_\sigma  \partial_\sigma \partial_{\sigma_0} D &= \frac{-1}{\sigma_0^3} \Big( \partial_\sigma^2 \sigma_B^2(\theta) + 2(\mu_B - \mu_0) \partial_\sigma^2 \mu_B(\theta) + 2\big(\partial_\sigma \mu_B(\theta) \big)^2 \Big).    
\end{align*} 
Thus, the Christoffel symbols $\Gamma_{ij,k}$ are given by
\begin{align*}
    \Gamma_{11,1} &= -\partial_\mu \partial_\mu \partial_{\mu_0} D_{|\theta=\theta_0}  = \frac{ \partial_\mu^2 \mu_B(\theta_0) }{\sigma_0^2},\\
    \Gamma_{21,1}&=\Gamma_{12,1} = -\partial_\sigma \partial_\mu \partial_{\mu_0} D_{|\theta=\theta_0}  =  \frac{ \partial_{\sigma\mu} ^2 \mu_B(\theta_0) }{\sigma_0^2},\\
    \Gamma_{22,1} &= -\partial_\sigma \partial_\sigma \partial_{\mu_0} D_{|\theta=\theta_0} = \frac{\partial_\sigma^2 \mu_B(\theta_0)}{\sigma_0^2},\\
    \Gamma_{11,2} &= -\partial_\mu \partial_\mu \partial_{\sigma_0} D_{|\theta=\theta_0} = \frac{1}{\sigma_0^3} \Big( \partial_\mu^2 \sigma_B^2(\theta_0) + 2(\mu_B(\theta_0) - \mu_0) \partial_\mu^2 \mu_B(\theta_0) + 2\big(\partial_\mu \mu_B(\theta_0) \big)^2 \Big),\\
\Gamma_{12,2} &=\Gamma_{21,2} = -\partial_\mu \partial_\sigma \partial_{\sigma_0} D_{|\theta=\theta_0}\\ 
&= \frac{1}{\sigma_0^3} \Big( \partial_\mu \partial_\sigma \sigma_B^2(\theta_0) + 2(\mu_B(\theta_0) - \mu_0) \partial_{\mu\sigma}^2 \mu_B(\theta_0) +2\big(\partial_\sigma \mu_B(\theta_0) \big)\big(\partial_\mu \mu_B(\theta_0) \big) \Big),\\
\Gamma_{22,2} &= -\partial_\sigma  \partial_\sigma \partial_{\sigma_0} D_{|\theta=\theta_0} = \frac{1}{\sigma_0^3} \Big( \partial_\sigma^2 \sigma_B^2(\theta_0) + 2(\mu_B(\theta_0) - \mu_0) \partial_\sigma^2 \mu_B(\theta_0) + 2\big(\partial_\sigma \mu_B(\theta_0) \big)^2 \Big).    
\end{align*} 

\paragraph{First and second order partial derivatives of $\mu_B$ and $\sigma_B^2$ with respect to $(\mu,\sigma)$}
The two conjugate Christoffel symbols $\Gamma_{ij,k}$ and $\Gamma_{ij,k}^*$ are functions of the first and second order partial derivatives of $\mu_B$ and $\sigma_B^2$. Now we explicitly compute these derivatives. To begin with, let us compute the first order partial derivatives of $\mu_B$:
\begin{align*}\partial_\mu \mu_B(\theta) &= 1 - \sigma^2 \big[ \partial_\mu q_\theta \big]_a^b,\\
    \partial_\sigma \mu_B(\theta) &= -2\sigma \big[ q_\theta \big]_a^b - \sigma^2 \big[ \partial_\sigma q_\theta \big]_a^b,\\
    \partial_\mu \sigma_B^2 (\theta) &= - \sigma^2  \big[- q_\theta +(x-\mu) \partial_\mu q_\theta \big]_a^b - 2 (\mu - \mu_B) \Big( 1 - \partial_\mu \big[\mu_B(\theta)\big] \Big),\\
    \partial_\sigma \sigma_B^2(\theta) &= 2\sigma \left( 1 - \big[ (x-\mu)q_\theta \big]_a^b \right) - \sigma^2\Big[ (x-\mu)\partial_\sigma q_\theta \Big]_a^b + 2(\mu-\mu_B)\partial_\sigma \mu_B(\theta). \\
\end{align*}
For the second order partial derivatives of $\mu_B$ and $\sigma_B ^2$, we obtain 
\begin{align*}
\partial_\mu ^2 \mu_B(\theta) &= -\sigma^2 \big[\partial_\mu ^2 q_\theta \big]_a ^b,\\
\partial_{\sigma\mu} ^2 \mu_B(\theta) &= -2 \sigma \big[ \partial_\mu q_\theta \big]_a ^b - \sigma^2 \big[ \partial_{\sigma\mu} ^2 q_\theta \big]_a ^b,\\
\partial_{\sigma} ^2 \mu_B(\theta) &= -2\big[q_\theta \big]_a ^b -4\sigma \big[\partial_\sigma q_\theta \big]_a ^b - \sigma^2 \big[\partial_{\sigma\mu} ^2 q_\theta \big]_a ^b,
\end{align*}
\begin{align*}
\partial_{\mu}^2 \sigma_B^2 &=-\sigma^2 \big[-2\partial_\mu q_\theta + (x-\mu) \partial_\mu ^2 q_\theta \big]_a^b -2\big(1-\partial_\mu \mu_B(\theta) \big)^2 + 2(\mu - \mu_B) \partial_\mu ^2 \mu_B, \\
\partial_{\sigma\mu}^2 \sigma_B^2 &=  -2\sigma\big[-q_\theta + (x-\mu) \partial_\mu q_\theta \big]_a^b - \sigma^2 \big[ -\partial_\sigma q_\theta + (x-\mu)\partial_{\sigma\mu} ^2 q_\theta \big]_a^b\\ 
&+ 2 \partial_\sigma \mu_B (1- \partial_\mu \mu_B) + 2\partial_{\sigma\mu} ^2  \mu_B (\mu - \mu_B),\\
\partial_{\sigma}^2 \sigma_B^2 &=  2\left(1-\big[(x-\mu) q_\theta \big]_a^b \right)- 4\sigma \big[(x-\mu) \partial_\sigma q_\theta \big]_a^b - \sigma^2 \big[(x-\mu) \partial_\sigma ^2 q_\theta \big]_a^b \\
&+ 2 (\partial_\sigma \mu_B )^2  - 2 (\mu - \mu_B) \partial_\sigma ^2 \mu_B.
\end{align*}

\paragraph{First and second order partial derivatives of $q_\theta$ with respect to $(\mu,\sigma)$}
We also need to compute the partial derivatives of $q_\theta$. For the first order, we get
\begin{align*}
\partial_\mu q_\theta(x) &= q_\theta(x) \left( \frac{x-\mu}{\sigma^2} + \big[q_\theta(x) \big]_a ^b  \right),\\
\partial_\sigma q_\theta(x) &= q_\theta(x) \left(  \frac{1}{\sigma}\big[ (x-\mu) q_\theta(x) \big]_a ^b -\frac{1}{\sigma} + \frac{(x-\mu)^2}{\sigma^3}  \right).    
\end{align*}
The computation of the second order derivatives follows
\begin{align*}
\partial_\mu ^2 q_\theta(x) &= \partial_\mu q_\theta(x) \left( \frac{x-\mu}{\sigma^2} + \big[q_\theta(x) \big]_a ^b  \right) + q_\theta(x) \left( \frac{-1}{\sigma^2} + \big[\partial_\mu q_\theta(x) \big]_a ^b  \right),\\
\partial_{\sigma\mu} ^2 q_\theta(x) &= \partial_\sigma q_\theta(x) \left( \frac{x-\mu}{\sigma^2} + \big[q_\theta(x) \big]_a ^b  \right) + q_\theta(x) \left( \frac{-2(x-\mu)}{\sigma^3} + \big[\partial_\sigma q_\theta(x) \big]_a ^b  \right),\\
\partial_\sigma ^2 q_\theta(x) &=  \partial_\sigma q_\theta(x) \left( \frac{1}{\sigma}\big[ (x-\mu) q_\theta(x) \big]_a ^b -\frac{1}{\sigma} + \frac{(x-\mu)^2}{\sigma^3}  \right) \\
&+ q_\theta(x) \left( \frac{-1}{\sigma}\big[ (x-\mu) q_\theta(x) \big]_a ^b 
+\frac{1}{\sigma}\big[ (x-\mu) \partial_\sigma q_\theta(x) \big]_a ^b +\frac{1}{\sigma^2} + \frac{-3(x-\mu)^2}{\sigma^4}  \right).
\end{align*} 

Let us summarize the previous computations:
\begin{enumerate}
    \item we computed the FIM of the truncated normal family on $[a,b]$;
    \item we expressed the Christoffel symbols as a function of the conditional mean $\mu_{[a,b]}$ and conditional variance $\sigma_{[a,b]}^2$ (that are functions of $\theta$) as well as their partial derivatives;
    \item we then gave the explicit expression of these partial derivatives as a function of the truncated density $q_\theta$ on $[a,b]$ and the partial derivatives of the latter;
    \item lastly, we explicitly computed the partial derivatives of $q_\theta$.
\end{enumerate}
This concludes the computation of the FIM and Christoffel symbols of the truncated normal family on $[a,b].$

\subsection{Gumbel family}
The Gumbel family is a location-scale family where the initial density $p$ is given by $p(x) = \exp\left(-x - e^{-x}\right).$
\subsubsection{Non-truncated case}
The FIM for the non-truncated version is given by
$$I_\theta = \frac{1}{s^2}\begin{bmatrix}
    1 & \gamma -1\\
    \gamma -1 & \beta + 1
\end{bmatrix},$$
where $\gamma \approx 0.5772...$ is the Euler-Mascheroni constant and $\beta$ is given by the integral
$$\beta = \int_0^\infty  \log(x)^2 x e^{-x}dx.$$
In this case, the Christoffel symbols can be easily computed. Indeed, they only depend on $s$. For the non-truncated Gumbel family, we numerically compute $\beta$ using the Gaussian quadrature method of SciPy\footnote{For documentation, see \href{https://docs.scipy.org/doc/scipy/tutorial/integrate.html}{https://docs.scipy.org/doc/scipy/tutorial/integrate.html}}.
\subsubsection{Truncated case}
The density of the truncated Gumbel distribution on $[a,b]$ writes $q_\theta(x) = \frac{1}{N_\theta} p_\theta(x) \textbf{1}_{x\in [a,b]},$ where $N_\theta$ is the normalizing constant
$$ N_\theta = \int_a^b \frac{1}{s^2} \exp\left(-\frac{x-m}{s} - e^{-\frac{x-m}{s}}\right) dx = F\left(\frac{b-m}{s} \right) - F\left(\frac{a-m}{s} \right),$$
and $F$ the cumulative distribution function of $p$ given by
$ F(x) = \exp(-e^{-x}).$
In this case, we computed the FIM for the truncated family using the formula
$$ \left(I_\theta^{[a,b]}\right)_{ij} = - \mathbb E_{X \sim q_\theta} \Big[ \partial_{\theta_i\theta_j}\log q_\theta(X) \Big],$$
where $(\theta_1,\theta_2) = (m,s)$. For $X \sim q_\theta$, we have
$$ \log q_\theta(X) = -2\log s - \frac{X-m}{s} - \exp\left( -\frac{X-m}{s} \right) - \log N_\theta.$$
Taking the Hessian w.r.t. $\theta=(m,s)$ and integrating against $q_\theta$ leads to
\begin{align*}
    (I_\theta^{[a,b]})_{ij} = -\partial_{\theta_i\theta_j} &2\log s - \int_a^b \partial_{\theta_i\theta_j} \left[\frac{x-m}{s}\right] q_\theta(x) dx \\
    &- \int_a^b \partial_{\theta_i\theta_j} \left[\exp\left( -\frac{x-m}{s} \right)\right] q_\theta(x) dx - \big(\text{Hess}_\theta \log N_\theta \big)_{ij}.
\end{align*} 
Notice that the first and forth term in the last equality can be explicitly computed. For the second and third ones, we need to resp. compute integrals of the form (after a variable substitution)
$$ \int_c^d y \exp\left( -y -e^{-y} \right)dy \ \ \ \text{ and } \ \ \ \int_c^d y^2 \exp\left( -2y -e^{-y} \right)dy.$$
Here, $c$ and $d$ depend on $a,b,m,s$. We computed these integrals using the Gaussian quadrature method of SciPy. 
For the Christoffel symbols, we used the standard finite difference method for differentiating the coefficients of $I_\theta^{[a,b]}$ w.r.t. $\theta$:
$$ \partial_{m} I_\theta^{[a,b]} \approx \frac{I_{\theta + (h,0)}^{[a,b]} - I_\theta^{[a,b]}}{h} \ \ \ \text{and} \ \ \ \partial_{s} I_\theta^{[a,b]} \approx \frac{I_{\theta + (0,h)}^{[a,b]} - I_\theta^{[a,b]}}{h},$$
where we took $h = 10^{-7}$. Of course, more advanced numerical differentiation methods can be used instead.

\subsection{Gamma family}
The pdf $p_\theta$ of the Gamma family writes
$$ p_\theta(x) = \frac{\beta^\alpha}{\Gamma(\alpha)}x^{\alpha-1}e^{-\beta x}, \ \ \ \theta = (\alpha,\beta) \in (0,\infty)^2.$$
Here, $\Gamma$ is the Euler integral of the second kind. In this framework we do not consider a truncated version of this family. The FIM $I_\theta$ for this family can be easily computed. Indeed, we have
$$ I_\theta = \begin{bmatrix}
    \psi'(\alpha) & -\frac{1}{\beta}\\
    -\frac{1}{\beta} & \frac{\alpha}{\beta^2}
\end{bmatrix}.$$
Here, the function $\psi$ is the digamma function. To numerically compute its first and second order derivatives, we used the previously discussed finite difference scheme with $h=10^{-7}$. 

\section{A few additional illustrative figures for the density perturbation method}

\begin{figure}[h]
    \centering
    \includegraphics[width=0.7\linewidth]{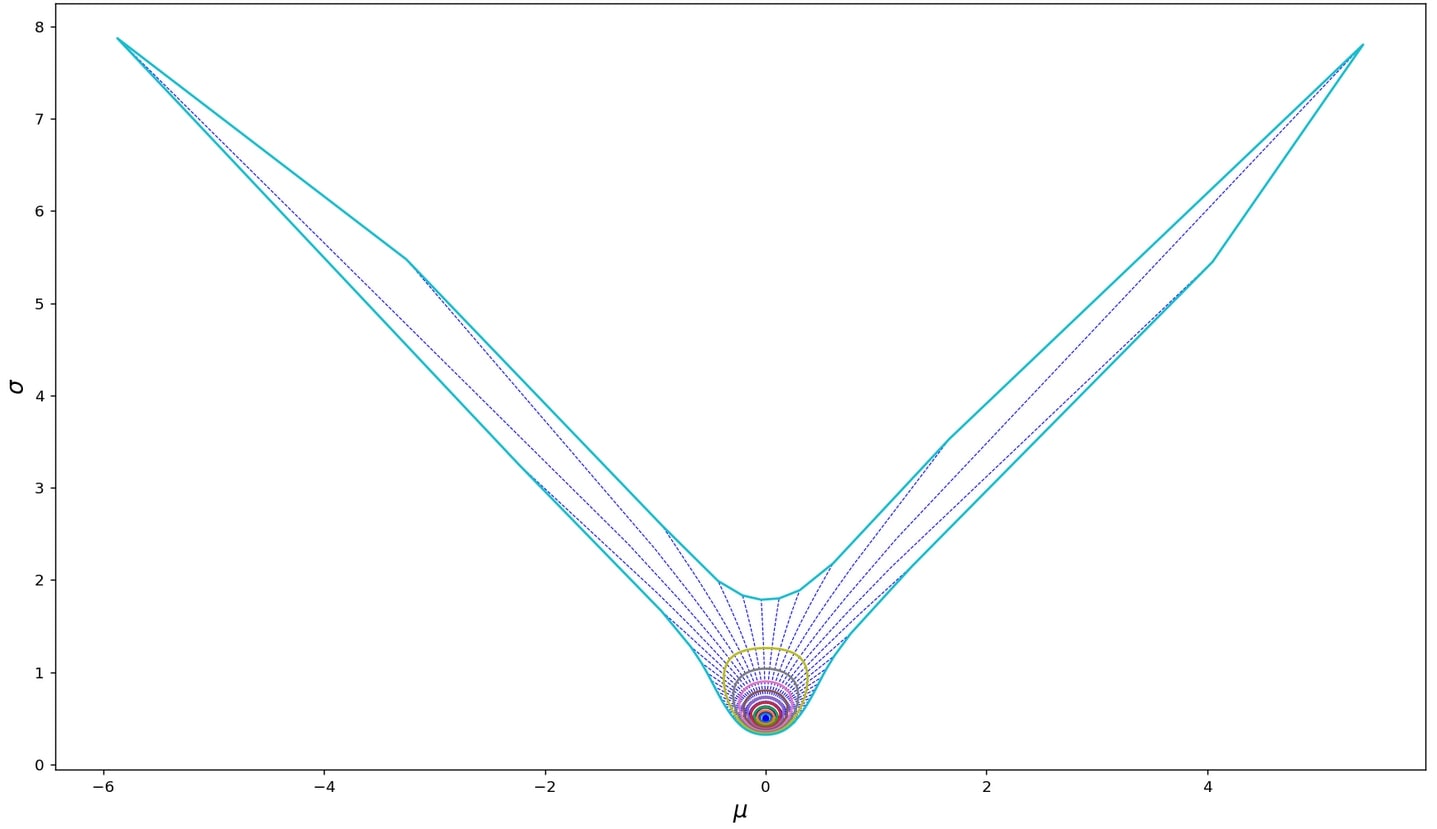}
    \caption{Concentric FR spheres for the truncated normal family on $[-2,2]$ in the $(\mu,\sigma)$ parameter space. The outermost sphere has radius $\delta_{\max}=0.5$ and is very distorted compared to the smaller ones.}
\end{figure}

\begin{figure}[h]
    \centering
    \includegraphics[width=0.7\linewidth]{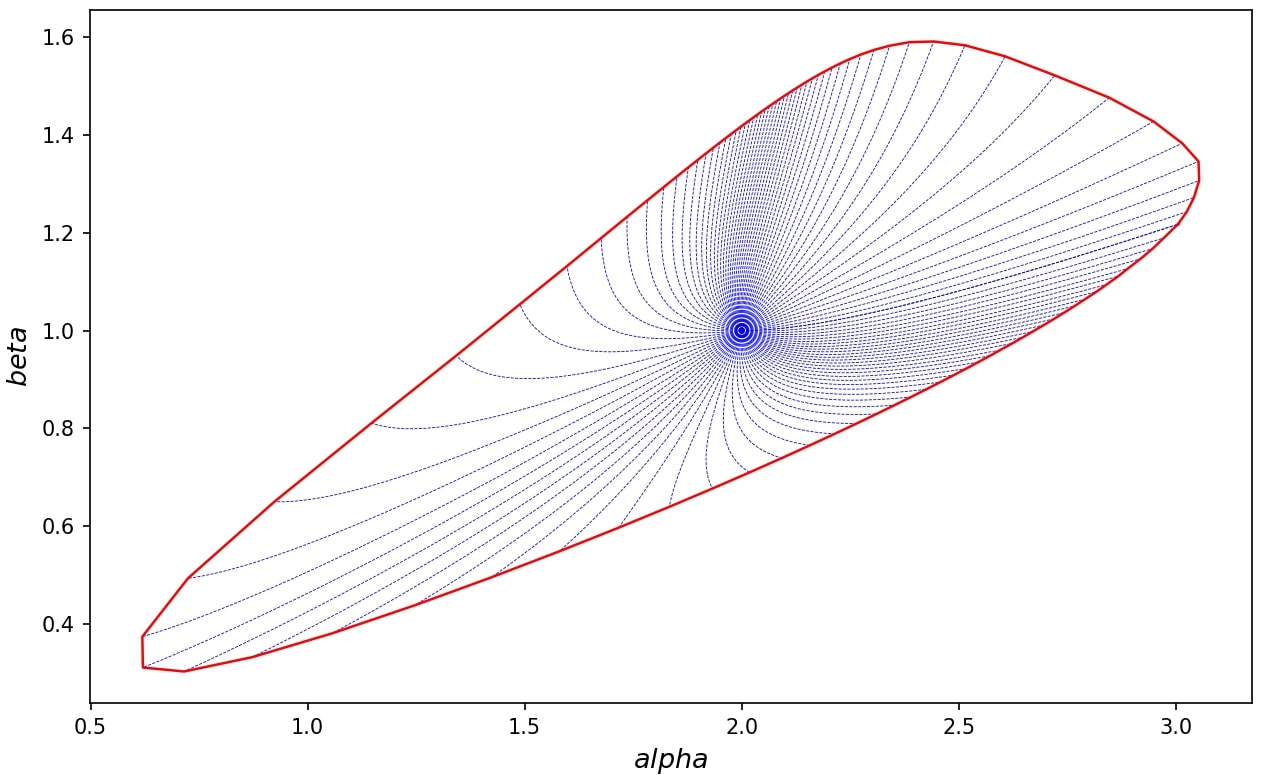}
    \caption{FR sphere (in red) and geodesics (in dashed blue) in the Gamma family in the $(\alpha,\beta)$ parameter space. The sphere is centered at $(2,1)$ and has radius $\delta=0.5$.}
\end{figure}

\begin{figure}[h]
    \centering
    \includegraphics[width=0.7\linewidth]{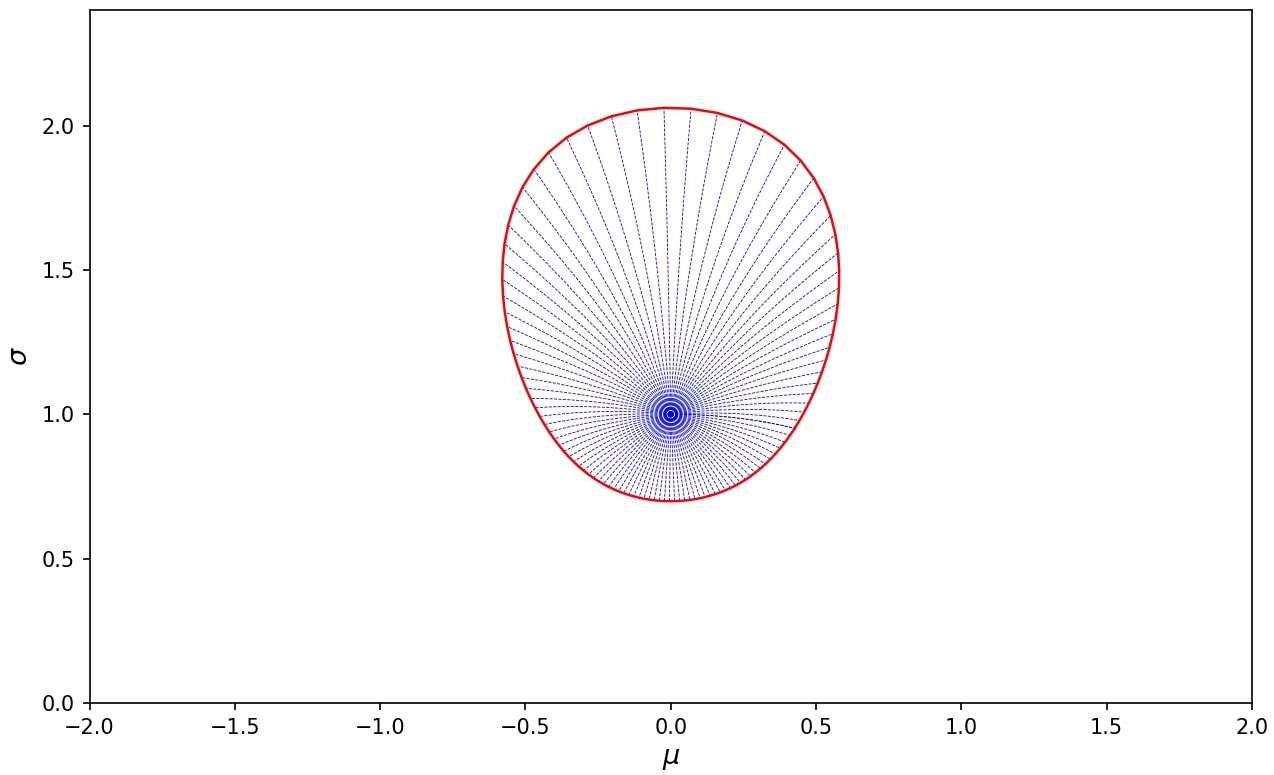}
    \caption{FR sphere (in red) and geodesics (in dashed blue) in the truncated normal family on $[-2,2]$. The sphere is centered at $(0,1)$ with radius $\delta=0.4.$}
\end{figure}

\end{appendices}

\bibliographystyle{plain}
\bibliography{biblio}

\begin{thebibliography}{10}

\bibitem{AJENJO2022102196}
A.~Ajenjo, E.~Ardillon, V.~Chabridon, B.~Iooss, S.~Cogan, and
  E.~Sadoulet-Reboul.
\newblock An info-gap framework for robustness assessment of epistemic
  uncertainty models in hybrid structural reliability analysis.
\newblock {\em Structural Safety}, 96:102196, 2022.

\bibitem{akahira2016second}
M.~Akahira.
\newblock Second-order asymptotic comparison of the {MLE} and {MCLE} of a
  natural parameter for a truncated exponential family of distributions.
\newblock {\em Annals of the Institute of Statistical Mathematics},
  68(3):469--490, 2016.

\bibitem{amari2000methods}
S.-I. Amari and H.~Nagaoka.
\newblock {\em Methods of information geometry}, volume 191.
\newblock American Mathematical Soc., 2000.

\bibitem{Bachoc2020}
F.~Bachoc, F.~Gamboa, M.~Halford, J.-M. Loubes, and L.~Risser.
\newblock {Explaining machine learning models using entropic variable
  projection}.
\newblock {\em Information and Inference: A Journal of the IMA}, 12(3), 2023.

\bibitem{Baudin2017}
M.~Baudin, A.~Dutfoy, B.~Iooss, and {A-L.} Popelin.
\newblock {Open TURNS: An industrial software for uncertainty quantification in
  simulation}.
\newblock In R.~Ghanem, D.~Higdon, and H.~Owhadi, editors, {\em Springer
  Handbook on {Uncertainty Quantification}}, pages 2001--2038. Springer, 2017.

\bibitem{benoumechiara2020detecting}
N.~Benoumechiara, N.~Bousquet, B.~Michel, and P.~Saint-Pierre.
\newblock Detecting and modeling critical dependence structures between random
  inputs of computer models.
\newblock {\em Dependence Modeling}, 8(1):263--297, 2020.

\bibitem{calin2014geometric}
O.~Calin and C.~Udri{\c{s}}te.
\newblock {\em Geometric modeling in probability and statistics}, volume 121.
\newblock Springer, 2014.

\bibitem{costa2015fisher}
S.~I. Costa, S.~A. Santos, and J.~E. Strapasson.
\newblock Fisher information distance: A geometrical reading.
\newblock {\em Discrete Applied Mathematics}, 197:59--69, 2015.

\bibitem{DaVeiga2021}
S.~Da~Veiga, F.~Gamboa, B.~Iooss, and C.~Prieur.
\newblock {\em Basics and Trends in Sensitivity Analysis. Theory and Practice
  in R}.
\newblock SIAM. Computational Science and Engineering, 2021.

\bibitem{defaux2018efficient}
G.~Defaux and G.~Perrin.
\newblock Efficient evaluation of reliability-oriented sensitivity indices.
\newblock {\em Journal of Scientific Computing}, 79(3):1433--1455, 2019.

\bibitem{dienstfrey2012uncertainty}
A.~M. Dienstfrey and R.~F. Boisvert.
\newblock {\em Uncertainty quantification in scientific computing}.
\newblock Springer, 2012.

\bibitem{do1992riemannian}
M.~P. Do~Carmo and J.~Flaherty~Francis.
\newblock {\em Riemannian geometry}, volume~2.
\newblock Springer, 1992.

\bibitem{Du2012}
X.~Du and Z.~Hu.
\newblock {First Order Reliability Method with Truncated Random Variables}.
\newblock {\em Journal of Mechanical Design}, 134(9):091005, 2012.

\bibitem{dytso2018structure}
A.~Dytso, H.~V. Poor, R.~Bustin, and S.~Shamai.
\newblock On the structure of the least favorable prior distributions.
\newblock In {\em 2018 IEEE International Symposium on Information Theory
  (ISIT)}, pages 1081--1085. IEEE, 2018.

\bibitem{efron1975defining}
B.~Efron.
\newblock Defining the curvature of a statistical problem (with applications to
  second order efficiency).
\newblock {\em The Annals of Statistics}, 3(6):1189--1242, 1975.

\bibitem{eguchi1985differential}
S.~Eguchi.
\newblock A differential geometric approach to statistical inference on the
  basis of contrast functionals.
\newblock {\em Hiroshima mathematical journal}, 15(2):341--391, 1985.

\bibitem{Gauchy2022}
C.~Gauchy, J.~Stenger, R.~Sueur, and B.~Iooss.
\newblock An information geometry approach to robustness analysis for the
  uncertainty quantification of computer codes.
\newblock {\em Technometrics}, 64(1):80--91, 2022.

\bibitem{halmos1949application}
P.~R. Halmos and L.~J. Savage.
\newblock Application of the {R}adon-{N}ikodym theorem to the theory of
  sufficient statistics.
\newblock {\em The Annals of Mathematical Statistics}, 20(2):225--241, 1949.

\bibitem{Hanea2022}
A.~M. Hanea, V.~Hemming, and G.~F. Nane.
\newblock Uncertainty quantification with experts: Present status and research
  needs.
\newblock {\em Risk Analysis}, 42(2):254--263, 2022.

\bibitem{Helton1996}
J.~Helton and D.~Burmaster.
\newblock Treatment of aleatory and epistemic uncertainty in performance
  assessments for complex systems.
\newblock {\em Reliability Engineering and System Safety}, 54:91--94, 1996.

\bibitem{hotelling1930spaces}
H.~Hotelling.
\newblock Spaces of statistical parameters.
\newblock {\em Bull. Amer. Math. Soc}, 36:191, 1930.

\bibitem{house2016bayesian}
T.~House, A.~Ford, S.~Lan, S.~Bilson, E.~Buckingham-Jeffery, and M.~Girolami.
\newblock Bayesian uncertainty quantification for transmissibility of
  influenza, norovirus and ebola using information geometry.
\newblock {\em Journal of the Royal Society Interface}, 13(121):20160279, 2016.

\bibitem{Hullermeier2021}
E.~Hüllermeier and W.~Waegeman.
\newblock Aleatoric and epistemic uncertainty in machine learning: an
  introduction to concepts and methods.
\newblock {\em Machine Learning}, 110(3):457--506, 2021.

\bibitem{idrissi2022quantileconstrained}
M.~Il~Idrissi, N.~Bousquet, F.~Gamboa, B.~Iooss, and J.-M. Loubes.
\newblock Quantile-constrained wasserstein projections for robust
  interpretability of numerical and machine learning models.
\newblock {\em Electronic Journal of Statistics}, 18:2721–2770, 2024.

\bibitem{iooss2022bepu}
B.~Iooss, V.~Verg{\`e}s, and V.~Larget.
\newblock {BEPU} robustness analysis via perturbed law-based sensitivity
  indices.
\newblock {\em Proceedings of the Institution of Mechanical Engineers, Part O:
  Journal of Risk and Reliability}, 236(5):855--865, 2022.

\bibitem{jeffreys1946invariant}
H.~Jeffreys.
\newblock An invariant form for the prior probability in estimation problems.
\newblock {\em Proceedings of the Royal Society of London. Series A.
  Mathematical and Physical Sciences}, 186(1007):453--461, 1946.

\bibitem{kass1989geometry}
R.~E. Kass.
\newblock The geometry of asymptotic inference.
\newblock {\em Statistical Science}, 4(3):188--219, 1989.

\bibitem{komaki2007bayesian}
F.~Komaki.
\newblock Bayesian prediction based on a class of shrinkage priors for
  location-scale models.
\newblock {\em Annals of the Institute of Statistical Mathematics},
  59:135--146, 2007.

\bibitem{kurniawan2022bayesian}
Y.~Kurniawan, C.~L. Petrie, K.~J. Williams, M.~K. Transtrum, E.~B. Tadmor,
  R.~S. Elliott, D.~S. Karls, and M.~Wen.
\newblock Bayesian, frequentist, and information geometric approaches to
  parametric uncertainty quantification of classical empirical interatomic
  potentials.
\newblock {\em The Journal of Chemical Physics}, 156(21):214103, 2022.

\bibitem{largau24}
V.~Larget and M.~Gautier.
\newblock Increasing conservatism in {BEPU IB LOCA} safety studies using
  complementary and industrially cost effective statistical tools.
\newblock In {\em Proceedings of Best Estimate Plus Uncertainty International
  Conference (BEPU 2024)}, Lucca, Italy, May 2024.

\bibitem{lehmann2006theory}
E.~L. Lehmann and G.~Casella.
\newblock {\em Theory of point estimation}.
\newblock Springer Science \& Business Media, 2006.

\bibitem{Lemaitre2015}
P.~Lema{\^\i}tre, E.~Sergienko, A.~Arnaud, N.~Bousquet, F.~Gamboa, and
  B.~Iooss.
\newblock Density modification-based reliability sensitivity analysis.
\newblock {\em Journal of Statistical Computation and Simulation},
  85(6):1200--1223, 2015.

\bibitem{ludtke2008information}
N.~L{\"u}dtke, S.~Panzeri, M.~Brown, D.~S. Broomhead, J.~Knowles, M.~A.
  Montemurro, and D.~B. Kell.
\newblock Information-theoretic sensitivity analysis: a general method for
  credit assignment in complex networks.
\newblock {\em Journal of The Royal Society Interface}, 5(19):223--235, 2008.

\bibitem{MAUMEDESCHAMPS2018122}
V.~Maume-Deschamps and I.~Niang.
\newblock Estimation of quantile oriented sensitivity indices.
\newblock {\em Statistics and Probability Letters}, 134:122--127, 2018.

\bibitem{nechval2016novel}
N.~A Nechval, S.~Prisyazhnyuk, and V.~F. Strelchonok.
\newblock {A Novel Approach to Finding Sampling Distributions for Truncated
  Laws Via Unbiasedness Equivalence Principle}.
\newblock {\em American Journal of Theoretical and Applied Statistics},
  9(1):40--48, 2016.

\bibitem{nielsen2020elementary}
F.~Nielsen.
\newblock An elementary introduction to information geometry.
\newblock {\em Entropy}, 22(10):1100, 2020.

\bibitem{PAPI2021701}
F.~Papi, F.~Balduzzi, G.~Ferrara, and A.~Bianchini.
\newblock Uncertainty quantification on the effects of rain-induced erosion on
  annual energy production and performance of a multi-mw wind turbine.
\newblock {\em Renewable Energy}, 165:701--715, 2021.

\bibitem{rao1945information}
C.~R. Rao.
\newblock Information and accuracy attainable in the estimation of statistical
  parameters. kotz s \& johnson nl (eds.), breakthroughs in statistics volume
  i: Foundations and basic theory, 235--248, 1945.

\bibitem{Shao2023}
Y.~Shao, K.~Li, T.~Zhang, Y.~J. Yang, and S.~Chu.
\newblock {Modeling of truncated normal distribution for estimating hydraulic
  parameters in water distribution systems: taking nodal water demand as an
  example}.
\newblock {\em Journal of Hydroinformatics}, 25(5):2053--2068, 2023.

\bibitem{shima2007geometry}
H.~Shima.
\newblock {\em The geometry of Hessian structures}.
\newblock World Scientific, 2007.

\bibitem{sternfels2013geometric}
H.~R. Sternfels.
\newblock {\em Geometric Tools, Sampling Strategies, Bayesian Inverse Problems
  And Design Under Uncertainty}.
\newblock PhD thesis, Cornell University, 2013.

\bibitem{Sueur2017}
R.~Sueur, B.~Iooss, and T.~Delage.
\newblock {Sensitivity analysis using perturbed-law based indices for quantiles
  and application to an industrial case}.
\newblock In {\em {10th International Conference on Mathematical Methods in
  Reliability (MMR 2017)}}, Grenoble, France, July 2017.

\bibitem{Ding2016}
D.~Tao, S.~N. Anfinsen, and C.~Brekke.
\newblock Robust {CFAR} detector based on truncated statistics in
  multiple-target situations.
\newblock {\em IEEE Transactions on Geoscience and Remote Sensing},
  54(1):117--134, 2016.

\bibitem{torre2019general}
E.~Torre, S.~Marelli, P.~Embrechts, and B.~Sudret.
\newblock A general framework for data-driven uncertainty quantification under
  complex input dependencies using vine copulas.
\newblock {\em Probabilistic Engineering Mechanics}, 55:1--16, 2019.

\bibitem{tukey1949sufficiency}
J.~W. Tukey.
\newblock Sufficiency, truncation and selection.
\newblock {\em The Annals of Mathematical Statistics}, pages 309--311, 1949.

\bibitem{Turhan2012}
B.~Turhan.
\newblock On the dataset shift problem in software engineering prediction
  models.
\newblock {\em Empirical Software Engineering}, 17:62--74, 2012.

\bibitem{yang2024decision}
J.~Yang.
\newblock Decision-oriented two-parameter fisher information sensitivity using
  symplectic decomposition.
\newblock {\em Technometrics}, 66(1):28--39, 2024.

\bibitem{yoshioka2023information}
M.~Yoshioka and F.~Tanaka.
\newblock Information-geometric approach for a one-sided truncated exponential
  family.
\newblock {\em Entropy}, 25(5):769, 2023.

\bibitem{zhang2022high}
Y.~Zhang, Y.~Wu, and H.~Xia.
\newblock High resolution coherent doppler wind lidar incorporating phase-shift
  keying.
\newblock {\em Journal of Lightwave Technology}, 40(22):7471--7478, 2022.

\end{thebibliography}

\end{document}